\documentclass[12pt]{amsart}
\usepackage{color}
\usepackage{amsmath,amsthm, amscd, amssymb, amsfonts,pb-diagram}
\usepackage[all,cmtip]{xy}
\usepackage{enumitem}
\numberwithin{equation}{section}

\calclayout

\newcommand{\End}{\mbox{\rm End\,}}
\theoremstyle{plain}

\newtheorem{theorem}{Theorem}[section]
\newtheorem{corollary}[theorem]{Corollary}
\newtheorem{proposition}[theorem]{Proposition}
\newtheorem{lemma}[theorem]{Lemma}

\theoremstyle{definition}
\newtheorem{definition}[theorem]{Definition}

\newtheorem{ejem}[theorem]{Example}

\theoremstyle{remark}
\newtheorem{remark}[theorem]{Remark}

\newcounter{commentcounter}

\newcommand{\So}{{\mathcal{S} }}
\newcommand{\Fot}{{\overset{F}{\otimes}}}
\newcommand{\N}{{\mathcal N}}
\newcommand{\C}{{\mathcal C}}
\newcommand{\D}{{\mathcal D}}
\newcommand{\M}{{\mathcal M}}

\newcommand{\Tt}{\ensuremath{{k^{\times}}}}
\newcommand{\END}{\mathcal{E}nd}
\newcommand\id{\operatorname{id}}
\newcommand\Aut{\operatorname{Aut}}

\newcommand{\ot}{\otimes}
\newcommand\ad{\operatorname{ad}}

\newcommand{\fde}{{\triangleright}}
\newcommand{\fiz}{{\triangleleft}}
\newcommand{\Z}{\mathcal{Z}}

\newcommand\op{\operatorname{op}}
\newcommand\Hom{\operatorname{Hom}}
\newcommand\Rep{\operatorname{Rep}}

\newcommand\Ind{\operatorname{Ind}}

\newcommand\vect{\operatorname{Vec}}

\newcommand\Out{\operatorname{Out}}
\newcommand{\TY}{{\mathcal {TY}}}
\newcommand\Ad{\operatorname{Ad}}
\newcommand\Inn{\operatorname{Inn}}
\newcommand\Stab{\operatorname{Stab}}
\newcommand\Inv{\operatorname{Inv}}
\newcommand\BrPic{\operatorname{BrPic}}
\newcommand\Fun{\operatorname{Fun}}

\begin{document}
\title{Tensor functors between Morita duals of fusion categories}
\author{C\'esar Galindo}
\email{cn.galindo1116@uniandes.edu.co}
\address{Departamento de Matematicas, Universidad de los Andes, Carrera 1 N. 18A -10, Bogot\'a, Colombia}
\author{Julia Yael Plavnik}
 \email{ julia@math.tamu.edu}
\address{Department of Mathematics,
    Texas A\&M University,
    College Station, TX
    U.S.A.}

\thanks{C.Galindo would like to thank the hospitality of the Mathematics Departments at MIT
where part of this work was carried out. This project began while J.Plavnik was at Universidad de Buenos Aires, and the support of that institution is gratefully acknowledged. C. Galindo  was partially supported by the FAPA funds from vicerrectoria de investigaciones de la Universidad de los Andes. J Plavnik was partially supported by CONICET, ANPCyT  and Secyt (UNC)} 
\date{}

\begin{abstract}
Given a fusion category $\C$ and an indecomposable $\C$-module category $\M$,
the fusion category $\C^*_\M$ of $\C$-module endofunctors of $\M$ is called the (Morita) dual fusion category of $\C$ with respect to $\M$. We describe tensor functors between two arbitrary duals $\C^*_\M$ and $\D^*_\N$ in terms of data associated to $\C$ and $\D$. We apply the results to $G$-equivariantizations of fusion categories and group-theoretical fusion categories. We describe the orbits of the action of the Brauer-Picard group on the set of module categories and we propose a categorification of the Rosenberg-Zelinsky sequence for fusion categories.
\end{abstract}

\maketitle

\setcounter{tocdepth}{1}

\section{Introduction}
{\bf 1.} In this paper $k$ will denote an algebraically closed field
of characteristic zero. By a \emph{fusion category} we mean a 
$k$-linear semisimple rigid tensor category with finitely many
isomorphism classes of simple objects and simple unit object {\bf
1.} For further reading we recommend \cite{BK, EGNO}.
\smallbreak
The set of isomorphism classes of invertible objects of a fusion
category $\C$ forms a group with multiplication induced by the
tensor product that we will denote by $\Inv(\C)$. A fusion
category is called \emph{pointed} if all its simple objects are
 invertible. A pointed fusion category $\C$ is
equivalent to $\vect_G^\omega$, that is, the category of  $G$-graded
finite dimensional vector spaces, where $G=\Inv(\C)$ and the
associativity constraint is given by a $3$-cocycle $\omega \in H^3(G,k^{\times})$.
\smallbreak
A very useful technique for the characterization of fusion
categories is the categorical Morita equivalence, \cite{Nik-surv}. Given a fusion category $\C$ and an indecomposable
left $\C$-module category $\M$, the fusion category
$\C_\M^*:=\END_\C(\M)$ is called the \emph{(Morita/categorical) dual} of $\C$ with respect to $\M$.
Important constructions in fusion category theory, such as the
Drinfeld's center and $G$-equivariantization of fusion categories, can be seen
as special cases of categorical duality. Also, fundamental examples
of fusion categories, such as group-theoretical  \cite{ENO,ostrik} and weakly
group-theoretical fusion categories \cite{ENO2}, are defined in terms of categorical duals.
\smallbreak
For pointed fusion categories, there is a very simple and concrete description of tensor functors and tensor natural transformations using group cohomology and group homomorphisms.
However, such a description for group-theoretical fusion categories (i.e.,  Morita duals of pointed fusion categories) is currently lacking\footnote{It
corresponds to problem 10.1 http://aimpl.org/fusioncat/10/ posted by
Shlomo Gelaki.}.

\medbreak {\bf 2.} The goal of this paper is to describe functors between two arbitrary duals, $\C^*_\M$ and $\D_\N^*$, in terms of data associated with $\C$ and $\D$. We  then apply the results to group-theoretical fusion categories and equivariantizations.
\smallbreak

Let $\mathfrak{Funct}$ be the category whose \textbf{objects} are pairs $(\C,\M)$, where $\C$ is a fusion
category and $\M$ is an indecomposable left $\C$-module category,
and \textbf{arrows} from $(\C,\M)$ to $(\D,\N)$ are equivalence
classes of monoidal functors from $\C^*_\M$ to $\D_\N^*$. The
\textbf{composition} of arrows is the equivalence class of the usual
composition of monoidal functors.

With the notation above, the category of 
group-theoretical fusion categories and equivalence classes of functors between them is equivalent to the
subcategory of  $\mathfrak{Funct}$ whose objects are pairs
$(\C,\M)$, with $\C$ a pointed fusion category \cite{ostrik}.


In order to describe $\mathfrak{Funct}$ in terms
of data associated with $\C$ and $\D$, we introduce the category
$\mathfrak{Cor}$:
\begin{itemize}
\item[(1)] \textbf{Objects} are pairs $(\C,\M)$, where $\C$ is a fusion category and $\M$ is an indecomposable left semisimple $\C$-module category.
\item[(2)] \textbf{Arrows} from $(\C,\M)$ to $(\D,\N)$ are equivalence classes of  pairs $(\So,
\alpha)$, where $\So$ is a $(\C,\D)$-bimodule category and
$\alpha:\So\boxtimes_\D\N \to \M$ is an equivalence of left
$\C$-module categories. Two pairs  $(\So, \alpha)$ and $(\So',
\alpha')$ represent the same arrow from $(\C,\M)$ to $(\D,\N)$  if
there exists a pair $(\phi,a)$, where $\phi:\So\to \So'$ is a
$(\C,\D)$-bimodule equivalence and $a$ is a natural isomorphism of
left $\C$-module functors from $\alpha$ to $\alpha'\circ
(\phi\boxtimes_{\D}\N)$, see the following diagram

\begin{eqnarray} \label{2-isomorphism condition for morphism in
pseudomoids}
 \xy
 (-12,0)*+{\So'\boxtimes_\D \N}="L";
 (12,0)*+{\mathcal{M}}="R";
 (0,16)*+{\So\boxtimes_\D \N}="T";
    {\ar_{\phi\boxtimes_{\D}\N} "T";"L"};
    {\ar^{\alpha} "T";"R"};
    {\ar_{\alpha'} "L";"R"};
    {\ar@{=>}^{a} (-2,2);(4,7)}
 \endxy.
\end{eqnarray}
\end{itemize}
If $(\So,\alpha)\in \mathfrak{Cor}((\C,\M),(\D,\N))$ and
$(\mathcal{P},\beta)\in \mathfrak{Cor}((\D,\N),
(\mathcal{L},\mathcal{T}))$ are arrows, the composition is
$$(\So,\alpha)\odot (\mathcal{P},\beta)= (\So\boxtimes_\D
\mathcal{P}, \alpha\odot \beta)\in \mathfrak{Cor}((\C,\M),
(\mathcal{L},\mathcal{T})),$$ where $\alpha\odot \beta$ is given by
the commutativity of the following diagram
$$
\begin{diagram}
\node{(\So\boxtimes_\D
\mathcal{P})\boxtimes_{\mathcal{L}}\mathcal{T}}
\arrow{s,l}{a_{\So,\mathcal{P},\mathcal{T}}} \arrow{e,t}{\alpha\odot \beta} \node{\M}\\
\node{\So\boxtimes_\D
(\mathcal{P}\boxtimes_{\mathcal{L}}\mathcal{T})}
\arrow{e,t}{\id_{\So}\boxtimes_\D \beta}\node{\So\boxtimes_\D \N}
\arrow{n,r}{\alpha}
\end{diagram}
$$
Our first main result is an explicit description of tensor functors between categorical duals in terms of $\mathfrak{Cor}$.
\begin{theorem}\label{description of Funct}
The category $\mathfrak{Funct}$ is (contravariantly) equivalent to the category
$\mathfrak{Cor}$. Moreover, at the level of objects the equivalence is given by the identity.
\end{theorem}
\begin{remark}
The categories  $\mathfrak{Funct}$ and $\mathfrak{Cor}$ are truncations of bicategories.
We expect that the category equivalence of Theorem \ref{description of Funct}
comes from a truncation of a biequivalence of bicategories.
\end{remark}
Since every fusion category is dual to itself, the description of tensor functors between an arbitrary pair of Morita duals is very general. In fact, such a description includes the classification of tensor functors between any pair of fusion categories. However, Theorem \ref{description of Funct} has nontrivial consequences, e.g.,  Theorems \ref{coro doble cosets}, \ref{teorema Equivalent Equivariantizations} and Corollary \ref{Corol centro gt}). In particular using Theorem \ref{description of Funct}, we get an implicit group-theoretical description of all tensor functors between group-theoretical fusion categories.

\medbreak {\bf 3.} Let $\C$ be a fusion category. The group of
equivalence classes of invertible $\C$-bimodule categories is called
the \emph{Brauer-Picard group} and is denoted $\text{BrPic}(\C)$.
The Brauer-Picard group of a fusion category is the fundamental group
of a categorical 2-group, 
$\underline{\underline{\text{BrPic}}}(\C)$, that parametrizes the
extensions of a fusion category $\C$ by finite groups, 
\cite{ENO3}. An interesting problem is to explicitly compute the Brauer-Picard groups for specific fusion categories. Some results of this type were obtained in \cite{ENO3,GS,M2,NR}.

Below we will describe some applications of Theorem
\ref{description of Funct} to the Brauer-Picard group of a fusion
category. \medbreak

{\bf 3.1} The first application of Theorem \ref{description of Funct} is a categorification of the Rosenberg-Zelinsky
exact sequence. Let $\C$ be a fusion category and consider the
abelian group of (isomorphism classes of) invertible objects
$\operatorname{Inv}(\mathcal{Z}(\C))$ of the Drinfeld's center
$\mathcal{Z}(\C)$ of $\C$. For every $\C$-module category $\M$ we
have a group homomorphism
\begin{align*}
s:\Inv(\mathcal{Z}(\C))\to \Aut_\C(\M).
\end{align*}
For each $(X,c_{X,-})\in \Inv(\mathcal{Z}(\C))$, we define $(s_X,\gamma)\in \Aut_\C(\M)$  by  $s_X(M):= X\otimes M$ and $\gamma_{V,M}: s_X(V\otimes M)\to
V\otimes s_X(M)$, where $\gamma_{V,M}:=c_{X,V}\otimes
\id_{M}$, for all $V\in \C$, $M\in \M$.

Recall that by Theorem \ref{description of Funct}, there is an (contravariant) equivalence $\mathcal{K}:\mathfrak{Funct}\to \mathfrak{Cor}$ that is the identity at the level of objects. Given an arrow $(\So,\alpha )\in \mathfrak{Cor}$, we denote by $\pi_1(\So,\alpha)$ the projection to the first component. With this notation our second main result is a categorification of the Rosenberg-Zelinsky exact sequence.
\begin{theorem}\label{RZ}
The sequence of groups
$$1\to\ker(s)\to \Inv(\mathcal{Z}(\C))\overset{s}{\to} \Aut_{\C}(\M) \overset{\operatorname{conj}_\M}{\to} \Aut_\otimes
(\C_\M^*)\overset{\pi_1\circ \mathcal{K}}{\to} \operatorname{BrPic}(\C),$$
is exact.
\end{theorem}
For the case  $\M= \C$, the exact sequences of Theorem \ref{RZ} can be rewritten as follows:
$$1\to\Aut_{\otimes}(\id_{\C})\to \Inv(\mathcal{Z}(\C))\overset{s}{\to} \Inv(\C)\overset{\operatorname{conj}_\C}{\to} \Aut_\otimes(\C)\overset{\pi_1\circ \mathcal{K}}{\to}  \operatorname{BrPic}(\C).$$

The image of $\pi_1\circ \mathcal{K}$ is denoted by
$\Out_\otimes(\C)$ and is called the \textit{group of tensor outer-autoequivalences of $\C$},  \cite{Ga2}. The image of $\operatorname{conj}_\C$ is denoted by 
$\operatorname{Inn}(\C)$ and is called \textit{the group of inner-autoequivalence of $\C$}.

Thus we have the exact sequence

\[1\to \operatorname{Inn}(\C) \to \Aut_\otimes(\C) \to \Out_\otimes(\C) \to 1.\]

In event that $\C$ has a braided structure the sequence is just an
inclusion $\Aut_{\otimes}(\C)\hookrightarrow \operatorname{BrPic}(\C)$, see
Corollary \ref{corolario RZ caso trenzado}.

\medbreak {\bf 3.2} 
We define $\mathcal{T}(\C)$ to be the set of equivalence classes of right
$\C$-module categories $\M$ such that $\C\cong \C_\M^*$ as fusion
categories. Note that the sets $ \Out_\otimes(\C) \setminus \operatorname{BrPic}(\C)$ and
$\mathcal{T}(\C)$ are right  $\operatorname{BrPic}(\C)$-sets in a natural way.

By \cite[Proposition 4.2]{ENO3} we have a map $$\operatorname{BrPic}(\C)\to
\mathcal{T}(\C),$$ given by forgetting the left $\C$-module structure.
This map factors through the left action  of
$\operatorname{Out}_\otimes(\C)$, thus we have a  map
$$U: \operatorname{Out}_\otimes(\C) \setminus \operatorname{BrPic}(\C) \to \mathcal{T}(\C).$$
\begin{theorem}\label{coro doble cosets}
\begin{itemize}
\item[(1)] The map $U$ is a bijective map of $\operatorname{BrPic}(\C)$-sets.
\item[(2)] $|\operatorname{BrPic}(\C)|=|\operatorname{Out}_\otimes(\C)| |\mathcal{T}(\C)|$.
\end{itemize}
\end{theorem}
As an application of Theorem \ref{coro doble cosets} we  compute the
Brauer-Picard group of some Tamabara-Yamagami categories, as well as some pointed fusion
categories with non-trivial associator.

Also we  present an algorithmic procedure to reduce the calculation
of the Brauer-Picard group to the computation of $\operatorname{Out}_\otimes(\C)$ and
some extra data that can be obtained using only a transversal of $ \mathcal{T}(\C) /\operatorname{Out}_\otimes(\C),$ see
Section \ref{Generalized crossed product}. \medbreak

{\bf 4.} \ In the remainder of the paper we focus on equivariantizations of fusion
categories and group-theoretical fusion categories. In particular, we  provide a description of all tensor equivalences between
equivariantizations of fusion categories. We also describe
invertible bimodule categories and their tensor products over
arbitrary pointed fusion categories. Our aim is to provide all
necessary ingredients for the application of Theorem
\ref{description of Funct} to any concrete example of
group-theoretical fusion categories.
\medbreak {\bf 5.}  \
The paper is organized as follows. In Section
\ref{preliminares} we  review module and bimodule categories over fusion categories.
In Section \ref{Morita revised} we will recall some tensor equivalences related with some dual categories and
the module structure induced by a tensor functor. In Section \ref{Seccion demotracion main theorem}
we will prove Theorem \ref{description of Funct}.
In Section \ref{Seccion equivalencis entre equivariantizaciones} we  give a
characterization of tensor equivalences between dual categories in terms of
certain data. We also study tensor functors between equivariantizations of fusion categories.
In particular, we give a description of tensor equivalences of Drinfeld's centers of pointed fusion categories.
In Section \ref{Seccion aplicaciones BrPic} we show the exactness of the Rosenberg-Zelinsky
sequence and give a proof of Theorem \ref{coro doble cosets}.
In Section \ref{Seccion bimodules de pointed fusion cat}
we review the theory of $G$-sets and study
module categories over pointed categories and their tensor products. We use this 
to describe tensor functors between group-theoretical fusion categories.
In Appendix \ref{Apendice} we give an explicit description of certain  2-subcategory of all module categories over a pointed fusion category that is useful for alternative proofs to some well known results about group-theoretical fusion categories.

\smallbreak\subsubsection*{Acknowledgements} We thank Paul Bruillard and the referee for useful comments.

\section{Preliminaries}\label{preliminares}

\subsection{Modules categories and Morita equivalence}

Let $\C$ be a fusion category over $k$. A \emph{module category}
over $\C$ is a semisimple category $\M$ together with a biexact
functor $\otimes : \C \times \M \to \M$ satisfying natural
associativity and unit axioms.

A module category $\M$ over $\C$ is called \emph{indecomposable} if
it is not equivalent to a nontrivial direct sum of module
categories.  We refer the reader to \cite{ostrik1} for further reading.

Let $\M$ and $\N$ be left (respectively, right) module categories
over $\C$. We will denote by $\Fun_\C(.\M,.\N)$ (respectively,
$\Fun_{\C}(\M.,\N.)$) the category whose objects are $\C$-module
functors from $\M$ to $\N$, and whose morphisms are natural module
transformations between these functors. If it is necessary we will use
dots in order to distinguish between left and right structures.  In the particular case
$\M = \N$, we will use the notation
$\END_\C(\M):=\Fun_\C(\M,\M)$. It follows from \cite[Theorem
2.15]{ENO} that $\END_\C(\M)$ is a fusion category if and only if $\M$ is an indecomposable $\C$-module category.

Let $\M$ be a  module category over $\C$.
The {\em (Morita) dual cate\-gory} $\C^*_{\M}$ of $\C$ with respect to $\M$
is the tensor category $\END_\C(\M)$.

Two fusion categories $\C$ and $\D$ are \emph{Morita equivalent} if
there exists an indecomposable $\C$-module category, $\M$, such that
$\D^{\operatorname{op}} \cong_{\otimes} (\C_\M^*)$. Recall that the opposite tensor category $\D^{\operatorname{op}}$ of a given tensor category $\D$ is obtained from $\D$ by reversing the tensor product.

\subsection{Bimodules categories}

Let $\C$ and $\D$ be fusion categories. A \emph{$(\C,\D)$-bimodule category} is
both, a left $\C$-module category and a right $\D$-module category
such that these two actions are compatible. Equivalently, a
$(\C,\D)$-bimodule category is a module category over
$\C\boxtimes\D^{\operatorname{op}}$, where $\boxtimes$ is the
Deligne tensor product of abelian categories, see \cite{D} for a precise definition of Deligne tensor product.

If $\M$ is a $(\C, \D)$-bimodule category then its opposite category
$\M^{\operatorname{op}}$ is a $(\D, \C)$-bimodule category with actions
 given by $X\odot M = M \otimes X^*$ and $M\odot Y= Y^*\otimes M $, for all  $X\in \C, Y\in \D, M\in \M$,  \cite[subsection
2.9]{ENO3}.

Let $\M$ and $\N$ be left and right $\C$-module categories, respectively, the
tensor product $\M\boxtimes_\C\N$ was defined in \cite{ENO3}.

\begin{remark}\label{remarks sobre producto tensorial}
Let $\M$ be a right $(\mathcal{L},\C)$-bimodule category, $\N$ a $(\C,\D)$-bimodule category,
and $\mathcal{P}$ a left $(\D,\mathcal{Q})$-module category. Then, by
\cite[Proposition 3.15]{Justin}, there is a canonical equivalence
$(\M\boxtimes_\C \N)\boxtimes_\mathcal{D} \mathcal{P} \cong
\M\boxtimes_\C (\N \boxtimes_\D \mathcal{P})$ of 
$(\mathcal{L},\mathcal{Q})$-bimodule categories. Hence we can use the notation $\mathcal{M}\boxtimes_\C \N
\boxtimes_\D \mathcal{P}$ without any ambiguity.
\end{remark}
A $(\C, \D)$-bimodule category, $\M$, is \emph{invertible} if there
exist bimodule equivalences
\begin{equation*}
\M^{\operatorname{op}}\boxtimes_\C \M \cong \D  \quad \operatorname{\and} \quad
\M\boxtimes_\D \M^{\operatorname{op}} \cong \C.
\end{equation*}

The \emph{Brauer-Picard group}  $\BrPic(\C)$ of a fusion category $\C$ is
the set of equivalence classes of invertible $\C$-bimodule categories with product given by $\boxtimes_\C$,  \cite{ENO3}.

\subsection{The Drinfeld center of a fusion category}

The \textit{ Drinfeld center}
$\mathcal{Z}(\C)$ of a monoidal category $\C$, is a braided monoidal category defined as follows. The objects of $\mathcal{Z}(\C)$ are pairs $(X, c_{X,-})$,
where $X \in \C$ and $c_{X,Y} : X \otimes Y \to Y \otimes X$ are
isomorphisms natural in $Y$ satisfying $c_{X, Y\otimes Z} = ( \id_Y\otimes c_{X,Z} )(c_{X, Y}\otimes \id_Z)$ and $c_{X,\bf{1}} = \id_X$, for
all $Y,Z \in \C$. A morphism $f:(X,c_{X,-})\to (Y,c_{Y,-})$ is a
morphism $f:X\to Y$ in $\C$ such that $( \id_W\otimes f)c_{X,W}=
c_{Y,W} (f\otimes \id_W)$ for all $W \in \C$.

The  Drinfeld center is a braided monoidal category with structure given as follows:
\begin{itemize}
  \item the tensor product is $ (X, c_{X,-}) \otimes (Y, c_{Y,-}) =
(X\otimes Y, c_{X\otimes Y,-})$, where
$$c_{X\otimes Y,
Z} = (c_{X,Z} \otimes \id_Y)(\operatorname{id}_X \otimes c_{Y,Z}) : X \otimes Y
\otimes Z \to Z \otimes X \otimes Y,$$ for all $Z \in \C$,
  \item the identity element is $(\mathbf{1}, c_{\mathbf{1},-})$, $c_{\mathbf{1},Z} = \id_Z$
  \item the braiding is given by the morphism $c_{X,Y}$.
\end{itemize}

If $\C$ is a fusion category, then $\mathcal{Z}(\C)$ is a non-degenerate braided fusion category, see \cite[Corollary 3.9]{DGNO}.

\section{Morita equivalence of fusion categories
revised}\label{Morita revised}

Let $\So$ be an $(\C, \D)$-bimodule category and  $X$ be an object in $\C$. Left multiplication by $X$ gives rise to a right $\D$-module endofunctor of $\So$ that we will denote $L(X)$. Thus we
have a tensor functor
$$L : \C \to \END_\D(\So.), \quad X\mapsto L(X).$$ Conversely, each
tensor functor, $F : \C \to \END_\D(\So.)$, defines a unique
$(\C,\D)$-bimodule category structure on $\So$. 

\begin{remark}
It was proved in \cite[Proposition 4.2]{ENO3} that $\So$ is an invertible $(\C,
\D)$-bimodule category if and only if the fuctor $L$ is a tensor
equivalence.
\end{remark}

Let $\M$ be a  left $\C$-module category. We will consider $\M$ as a
left $\END_\C(\M)$-module category with action given by $$F\odot M=
F(M),$$ for $F\in \END_\C(\M), M\in \M$.

We denote by $\mathfrak{M}_r(\C)$ the $2$-category of right
$\C$-module categories and by $\mathfrak{M}_l(\END_\C(\M))$ the
2-category of left $\END_\C(\M)$-module categories. We define the $2$-functor
\begin{align*}
    \mathfrak{R}:\mathfrak{M}_l(\END_\C(\M))&\to \mathfrak{M}_r(\C)\\
    \So &\mapsto  \Fun_{\END_\C(\M)}(\M,\So).
\end{align*}
Notice that $\Fun_{\END_\C(\M)}(\M,\So)$ is indeed a right
$\C$-module category because the left actions of $\C$ and
$\END_\C(\M)$ commute, and the right action is given by
$$(F\odot X)(M)=F(X\otimes M),$$ for all $F\in
\Fun_{\END_\C(\M)}(\M,\So), X\in \C$ and $M\in \M$.

We also define the $2$-functor $\mathfrak{L}:\mathfrak{M}_r(\C)\to
\mathfrak{M}_l(\END_\C(\M))$ by
$$\mathfrak{L}(\N)=\N\boxtimes_\C\M,$$ for all $\N \in \mathfrak{M}_r(\C)$. Since the left
actions of $\C$ and $\END_\C(\M)$ on $\M$ commute in a coherent way,
$\mathfrak{L}(\N)$ is a left $\END_\C(\M)$-module category.

\begin{theorem}\label{epsilon equivalencia} Let $\M$ be a left $\C$-module category.
There is an equivalence of left $\END_\C(\M)$-module categories given by
\begin{align*}
\varepsilon:\Fun_{\END_\C(\M)}(\M,\So)\boxtimes_\C\M &\to \So\\
F\boxtimes_\C M &\mapsto F(M),
\end{align*}
for all $\So\in \mathfrak{M}_l(\END_\C(\M))$, $M\in \M$ and $F\in
\Fun_{\END_\C(\M)}(\M,\So)$.
\end{theorem}
\begin{proof}
The $2$-functors $\mathfrak{L}$ and $\mathfrak{R}$ introduced above
are adjoint 2-functors. The unit of the adjunction is the natural
$2$-transformation $\eta:\id_{\mathfrak{M}_r(\C)}\to
\mathfrak{R}\circ\mathfrak{L},$ given by
\begin{align*}
\eta:\N &\to \Fun_{\END_\C(\M)}(\M,\N\boxtimes_\C\M)\\
N &\mapsto (M\mapsto N\boxtimes M)
\end{align*}
and the counit of the adjunction is the natural $2$-transformation
$\varepsilon:\mathfrak{L}\circ \mathfrak{R}\to
\id_{\mathfrak{M}_l(\END_\C(\M))},$ given by
\begin{align*}
\varepsilon:\Fun_{\END_\C(\M)}(\M,\So)\boxtimes_\C\M &\to \So\\
F\boxtimes_\C M &\mapsto F(M),
\end{align*}
for all $\N\in \mathfrak{M}_r(\C), \So\in
\mathfrak{M}_l(\END_\C(\M))$, $N\in \N, M\in \M, F\in
\Fun_{\END_\C(\M)}(\M,\So)$.

Etingof and Ostrik proved that the $2$-functor $\mathfrak{R}$ is an
equivalence of $2$-categories \cite{EO}. Therefore, its left adjoint
$\mathfrak{L}$ is also an equivalence of $2$-categories and the unit
$\eta$ and the counit $\varepsilon$ of the adjunction are equivalences of module
categories.
\end{proof}
An important fact used in the proof of Theorem \ref{epsilon
equivalencia} is that the $2$-functor $\mathfrak{R}$ is a
biequivalence of $2$-categories \cite{EO}. Some useful results follow
from it.
\begin{remark}\label{Remark equivalencia Morita}
\begin{enumerate}[leftmargin=*]
\item The biequivalence $\mathfrak{R}$ gives at the level of morphisms, a canonical tensor equivalence
\begin{align*}\label{equivalencia importante}
\END_{\END_\C(.\M)}(.\So)&\to \END_\C(\Fun_{\END_\C(.\M)}(.\M,.\So).)\\
F&\mapsto (G\mapsto F\circ G),
\end{align*}for each $\So\in \mathfrak{M}_l(\END_\C(\M))$.
\item For a fix a  right $\C$-module category $\N$, the 2-functor
\begin{align*}
    \M.\mapsto \Fun_\C(\N.,\M.)
\end{align*}
defines a $2$-equivalence between the $2$-category
$\mathfrak{M}_r(\C)$ and the $2$-category
$\mathfrak{M}_l(\END_\C(\N))$. In particular, we have a natural
tensor equivalence
\begin{equation}\label{equivalencia importante}
\END_{\END_\C(\N.)}(.\Fun_\C(\N.,\M.))\to \END_\C(\N.).
\end{equation}
\end{enumerate}
\end{remark}

\subsection{Duality between tensor
functors}\label{subsection-duality-functors} In this subsection we
collect some definitions and results from \cite{ENO} that will be
useful later.

Let $\C$ and $\D$ be fusion categories and $(F,F^0):\C\to \D$ be a
tensor functor. Any $\D$-module category
$(\M,\otimes,\alpha)$ also has a $\C$-module structure induced by
the tensor functor $F$ that we denote by $(\M^F,\Fot,\alpha^F)$. Here
$\M^F=\M$ as an abelian category, the left action is defined by $V\ \Fot M:=
F(V)\otimes M$ and the associativity constraint is given by
$$\alpha^F_{V,W,M}:=
\alpha_{F(V),F(W),M}\circ(F_{V,W}^0\otimes\id_M):(V\otimes W)\ \Fot
M\to V\ \Fot (W\ \Fot M),$$ for all $V,W \in \C,$ $M\in\M^F$. 

There is also an associated dual tensor functor $F^*: \END_\D(\M)\to
\END_\C(\M^F)$. The dual functor $F^*$ is defined in the following
way: given a $\D$-module endofunctor $(T,c):\M\to \M$, we set
$$F^*(T)=T  \text{  and  } F^*(c)_{V,M}=c_{F(V),M},$$ for all $V\in
\C$, $M\in \M^F$.
\begin{remark}\label{dualidad de F} Let $F_1, F_2: \C\to \D$ be
tensor functors and let $\theta: F_1\overset{\sim}{\Rightarrow} F_2
$ be a monoidal natural isomorphism.
\begin{enumerate}[leftmargin=*]
\item If $\M$ is a left $\D$-module category, then the natural
transformation $\theta$ induces an isomorphism of $\C$-module
categories $\widetilde{\theta}=(\operatorname{Id}_\M, \overline{\theta}):
\M^{F_1} \to \M^{F_2}$, where
$\overline{\theta}_{X,M}=\theta_X\otimes \id_M$, $\forall  X\in \C,$
$M\in \M$.
\item  As tensor functors $(F^*)^*=F$ and
 as $\D$-module categories $(\M^{F})^{F^*}\cong\M$.
\end{enumerate}
\end{remark}
Now, we introduce a tensor equivalence between dual categories that will be useful later.

Let $\alpha:\M\to\N$ be an equivalence of $\C$-module categories,
with quasi-inverse $\alpha^*:\N\to \M$ and natural isomorphism
$\Delta: \id_\N\to \alpha\circ\alpha^*$. 

The functor
\begin{align}
\ad_\alpha: \END_\C(\M) &\to \END_\C(\N) \label{functor ad} \\
F &\mapsto \alpha\circ F\circ\alpha^* \notag
\end{align}
defines a tensor equivalence with structural natural isomorphism
\begin{equation*}
c_{F,G}=\alpha\circ F(\Delta_{G(-)}):\ad_\alpha(F\circ G) \to
\ad_\alpha(F)\circ \ad_\alpha(G).
\end{equation*}

\section{Proof of Theorem \ref{description of Funct}}\label{Seccion demotracion main theorem}

The  proof  of Theorem \ref{description of Funct} is consequence of several lemmas. First we will prove that $\mathfrak{Cor}$ is a category. After that we define a contravariant functor $\mathcal{K}:\mathfrak{Funct}\to \mathfrak{Cor}$ that is the identity at the level of objects. Finally, we prove that $\mathcal{K}$ is faithful and full, and thus  by \cite[Theorem 1,
p.91]{MacLane-book} an equivalence.

The following lemma shows that $\mathfrak{Cor}$ is in fact a category.

\begin{lemma}\label{Lema1}The relation described in $(2)$ in the definition of the category $\mathfrak{Cor}$ is an equivalence relation. The composition of arrows in $\mathfrak{Cor}$
does not depend on the representative of the equivalence class chosen. Moreover, the composition is associative.
\end{lemma}
\begin{proof}
A straightforward calculation shows that the relation described
above between the arrows is reflexive and symmetric.

The relation is also transitive and, therefore, it is an equivalence
relation. In fact, let $(\So, \alpha)$, $(\So', \alpha')$ and $(\So^{''}, \alpha^{''})$
are arrows from $(\C, \M)$ to $(\D, \N)$ in the category
$\mathfrak{Cor}$.  Assume that $(\So, \alpha)$ is related to $(\So^{'},
\alpha')$ via $(\phi, a)$ and $(\So^{'}, \alpha')$ is related to
$(\So^{''}, \alpha^{''})$ via $(\phi^{'}, a')$. Then $(\So, \alpha)$
is related to $(\So^{''}, \alpha^{''})$ via the equivalence
$\phi^{'}\circ\phi: \So\rightarrow \So^{''}$ and the natural
isomorphism $a'\circ  a$ from $\alpha$ to $\alpha^{''}\circ
(\phi^{'}\circ\phi\boxtimes_{\D}\id_{\N})$.

To see that the composition is well defined, let $(\So,
\alpha)$ and $(\So^{'}, \alpha')$ be arrows from $(\C, \M)$ to $(\D,
\N)$ in $\mathfrak{Cor}$ related by the pair $(\phi, a)$. Let
$(\mathcal{P}, \beta)$ and $(\mathcal{P}^{'}, \beta')$ be arrows
from $(\D, \N)$ to $(\mathcal{L}, \mathcal{T})$ in $\mathfrak{Cor}$
related by the pair $(\varphi, b)$. Recall that by Remark \ref{remarks sobre producto tensorial} the notation
$\mathcal{M}\boxtimes_\C \N \boxtimes_\D \mathcal{P}$ will yield no ambiguity. Then we may
assume that the associativity $1$-isomorphism is the identity. We want to see that $(\So,\alpha)\odot (\mathcal{P},\beta)=
(\So\boxtimes_{\D} \mathcal{P}, \alpha\odot \beta)$ and
$(\So^{'},\alpha^{'})\odot (\mathcal{P},\beta)=
(\So^{'}\boxtimes_{\D} \mathcal{P}, \alpha^{'}\odot \beta)$ are
related arrows from $(\C,\M)$ to $(\mathcal{L},\mathcal{T})$ in
$\mathfrak{Cor}$. Clearly, $\phi\boxtimes_{\D} \id_{\mathcal{P}}:
\So\boxtimes_\D \mathcal{P}\to \So'\boxtimes_\D \mathcal{P}$ is an
equivalence of $(\C, \mathcal{Q})$-bimodules. In addition, it holds
that
\begin{align*}(\alpha^{'}\odot t) \circ (\phi \boxtimes_{\D}
\id_{\mathcal{P}}\boxtimes_{\mathcal{L}} \id_{\mathcal{T}}) & =
\alpha^{'}\circ(\id_{\So^{'}}\boxtimes_{\D} t) \circ (\phi
\boxtimes_{\D} \id_{\mathcal{P}\boxtimes_{\mathcal{L}}
\mathcal{T}})\\
& = \alpha^{'} \circ (\phi \boxtimes_{\D}\id_{\N}) \circ
(\id_{\So}\boxtimes_{\D} t)\\ & \cong \alpha \circ
(\id_{\So}\boxtimes_{\D} t)\\ & = (\alpha\odot t).
\end{align*}
Similarly, $(\So,\alpha)\odot
(\mathcal{P},t)= (\So\boxtimes_{\D} \mathcal{P}, \alpha\odot t)$ and
$(\So,\alpha)\odot (\mathcal{P}^{'},t^{'})= (\So\boxtimes_{\D}
\mathcal{P}^{'}, \alpha\odot t^{'})$ are related arrows from
$(\C,\M)$ to $(\mathcal{L},\mathcal{T})$ in $\mathfrak{Cor}$. Thus
the composition of arrows is a well defined operation. 

Next, we will prove that the composition is associative. Let
$(\So,\alpha): (\C,\M)\to (\D,\N)$, $(\mathcal{P},\beta): (\D,\N)\to
(\mathcal{L},\mathcal{T})$ and $(\mathcal{R},\gamma):
(\mathcal{L},\mathcal{T})\to (\mathcal{A},\mathcal{Q})$ be arrows in
$\mathfrak{Cor}$. We have that:
\begin{align*}((\So,\alpha)\odot (\mathcal{P},\beta)) \odot (\mathcal{R},\gamma) & =
(\So\boxtimes_{\D}\mathcal{P}\boxtimes_{\mathcal{L}}\mathcal{R},
(\alpha\odot \beta)\circ (\id_{\So\boxtimes_{\D}\mathcal{P}}\boxtimes_{\mathcal{L}} \gamma))\\
& = (\So\boxtimes_{\D}\mathcal{P}\boxtimes_{\mathcal{L}}\mathcal{R},
\alpha\circ (\id_{\So}\boxtimes_{\D}\beta)\circ (\id_{\So}\boxtimes_{\D}\id_{\mathcal{P}}\boxtimes_{\mathcal{L}} \gamma))\\
& = (\So\boxtimes_{\D}\mathcal{P}\boxtimes_{\mathcal{L}}\mathcal{R},
\alpha\odot (\beta\circ(\id_{\mathcal{P}}\boxtimes_{\mathcal{L}} \gamma)))\\
& = (\So,\alpha)\odot ((\mathcal{P},\beta) \odot
(\mathcal{R},\gamma)),
\end{align*}
as we desired.
\end{proof}

We may assume, without loss of generality, that every fusion category
and every module category are simultaneously strict \cite[Proposition
2.2]{Ga1}. \smallbreak

Let $F\in \mathfrak{Funct}((\D,\N),(\C,\M))$, that is,
$F:\END_{\D}(.\N)\to \END_{\C}(.\M)$ is a tensor functor. The
category
$$\So_F:=\Fun_{\END_\D(.\N)}(.\N,.\M^F),$$ is a $(\C,\D)$-bimodule category with left
$\C$-action given by
\begin{align*}
\odot:\C\times \So_F &\to \So_F\\
(X,G) &\mapsto (X\odot G)(N) = X\otimes G(N),
\end{align*}
and right $\D$-action given by
\begin{align*}
\odot: \So_F\times \D &\to \So_F\\
(G,Y) &\mapsto (G\odot Y)(N) = G(Y\otimes N).
\end{align*}
The associativity constraints of the actions are induced by the
associativity constraint of the fusion categories acting on $\So_F$.
Since we have assumed that the fusion categories
are strict, it follows that $\So_F$ is a right and left strict module
category over $\D$ and $\C$, respectively.

Next we check that the right and left operations defined above
are actually a right $\D$-action and a left $\C$-action on $\So_F$.
Notice that if $X\in \C$ and $G\in\So_F$ then $X\odot G$ is also in
$\So_F = \Fun_{\END_\D(.\N)}(.\N,.\M^F)$. In fact, for $N\in\N$ and
$\varphi\in\END_\D(.\N)$, it follows that
\begin{align*}\varphi  \Fot ((X\odot G)(N))& = F(\varphi) \otimes (X \otimes G(N))\\
& = F(\varphi)((X \otimes G(N)))\\
& \cong X \otimes F(\varphi)(G(N))\\ & = (X\otimes (\varphi\Fot G(
N))\\ & = (X\odot G)(\varphi \otimes N).
\end{align*}

In addition, if $Y\in \C$ then 
\begin{align*}
    (X\odot(Y\odot G))(N) &= X\otimes
(Y\odot G)(N) \\ &= X\otimes (Y\otimes G(N)) \\ &\cong (X\otimes Y)\otimes
G(N) = ((X\otimes Y)\odot G)(N).
\end{align*}
It is clear that the unit of $\C$ is well behaved with respect to $\odot$.  It similarly follows that  the remaining operation gives  a right $\D$-action on $\So_F$. Moreover, a straightforward calculation shows that the  category $\So_F$ is a $(\C,\D)$-bimodule category.

By Theorem \ref{epsilon equivalencia}, the evaluation functor
$$\varepsilon:\So_F\boxtimes_\D\N \to \M$$ is an equivalence of
left $\C$-module categories.  We  define $\mathcal{K}((\C, \M)) := (\C, \M)\in \mathfrak{Cor},$ for
an object $(\C, \M)\in\mathfrak{Funct}$ and
$$\mathcal{K}(F):= (\So_F, \varepsilon)\in
\mathfrak{Cor}((\C,\M),(\D,\N)),$$ for  $F\in \mathfrak{Funct}((\D,\N),(\C,\M))$.

The pair $(\So_F, \varepsilon)$ does not depend on the equivalence
class of $F$, that is, if $G: \END_{\D}(.\N)\to \END_{\C}(.\M)$ is a
monoidal functor equivalent to $F$ then the arrows $(\So_F,
\varepsilon)$ and $(\So_G, \varepsilon)$ are related in
$\mathfrak{Cor}$. In fact, it follows from Remark \ref{dualidad de F} (1) that
the equivalence $F\cong G$ induces an
equivalence $\M^F \cong \M^G$.
Therefore, $\So_F$ and $\So_G$ are equivalent as $(\C,
\D)$-bimodules categories with the equivalence given by the identity
functor equipped with some natural isomorphism and  $\varepsilon =
\varepsilon \circ (\id_{\So_F}\boxtimes_{\D} \id_{\N})$.

Thus, we have defined an assignment $\mathcal{K}$ that sends
an arrow $F\in \mathfrak{Funct}((\D,\N), (\C,\M))$ to an arrow
$(\So_F, \varepsilon)\in \mathfrak{Cor}((\C,\M),(\D,\N))$.

\begin{lemma}  $\mathcal{K}$ is a contra\-va\-riant functor.
\end{lemma}
\begin{proof}
Let $F:\END_\D(\N)\to \END_\C(\M)$ and let $G:\END_\C(\M)\to
\END_{\mathcal{Q}}(\mathcal{L})$ be tensor functors and let $G\circ
F:\END_\D(\N)\to \END_{\mathcal{Q}}(\mathcal{L})$ be their
composition. Consider the corresponding objects in $\mathfrak{Cor}$
associated to these three functors by $\mathcal{K}$, $(\So_F =
\Fun_{\END_\D(.\N)}(.\N,.\M^F),\varepsilon)$, $(\So_G =
\Fun_{\END_\C(.\M)}(.\M,.\mathcal{L}^G),\varepsilon)$ and
$(\So_{G\circ F}$ $ = \Fun_{\END_{\D}(.\N)}(.\N,.\mathcal{L}^{G\circ
F}),\varepsilon)$.
The composition induces a $(\mathcal{Q},\D)$-bimodule functor
$$\phi: \So_G\boxtimes_{\C}\So_F\to \So_{G\circ F}.$$

The commutativity of the diagram
$$
\begin{diagram}
\node{\So_G\boxtimes_\C
\So_F\boxtimes_\D\N}\arrow{e,t}{\id_{\So_G}\boxtimes_\C\varepsilon}
\arrow{s,l}{\phi\boxtimes_\D\N}\node{\So_G\boxtimes_\C\M}\arrow{s,r}{\varepsilon}
 \\
\node{\So_{G\circ F}\boxtimes_\D\N}
\arrow{e,t}{\varepsilon}\node{\mathcal{L}}
\end{diagram}
$$
implies that $\phi\boxtimes_\D\N$ is an equivalence, since the
evaluation maps $\varepsilon$ are equivalences, by Theorem
\ref{epsilon equivalencia}.

It follows from the proof of Theorem \ref{epsilon equivalencia} that
the functor $\mathfrak{L} = (-)\boxtimes_{\D}\N$ defines an
equivalence of 2-categories. Hence, $\phi$ is an equivalence of
$(\mathcal{Q},\D)$-bimodule categories. Thus, $\mathcal{K}(G\circ
F)$ and $\mathcal{K}(G)\odot \mathcal{K}(F)$ are related arrows in
$\mathfrak{Cor}$ by the pair ($\phi: \So_G\boxtimes_\C\So_F\to
\So_{G\circ F}, \varepsilon)$. Therefore $\mathcal{K}$ is a
contravariant functor, as we asserted.
\end{proof}

\begin{lemma}\label{Lemma K is faithful and full}
The functor $\mathcal{K}$ is faithful and full.
\end{lemma}
\begin{proof}
As in Section \ref{Morita revised},  
given $(\So,\alpha)\in \mathfrak{Cor}((\C,\M),(\D,\N))$, there is a
tensor functor associated 
$$L:\C\to \END_\D(\So.).$$
By the equivalence \eqref{equivalencia importante}, we can regard $L$
as a tensor functor $$L:\C\to  \END_{\END_\D(.\N)}(.\So\boxtimes_\D \N).$$
Dualizing $L$, we obtain $$L^*:\END_\D(.\N)\to
\END_\C(.(\So\boxtimes_\D \N)^L).$$ Now, using the functor
$\ad_\alpha$ introduced in  \eqref{functor ad}, we define the tensor functor
$$\mathcal{K}^{-1}((S,\alpha)):=\ad_\alpha \circ L^*:\END_\D(.\N)\to \END_\C(.\M).$$
An explicit description of the previous correspondence is the following. Given $(S,\alpha)\in \mathfrak{Cor}((\C,\M),(\D,\N))$, 
\begin{align*}
\mathcal{K}^{-1}((S,\alpha)):\END_\D(.\N)&\to \END_\C(.\M)\\
F &\mapsto \ad_\alpha(\operatorname{Id}_\So\boxtimes_\D F).
\end{align*}
Note that by construction these two assignments are mutually
inverse. Hence, it suffices to show  that $\mathcal{K}^{-1}$ is well defined. 

If  we have
two equivalent arrows $(\So,\alpha)$ and  $(\So',\alpha')$ related by
the pair $(\phi,a)$ in $\mathfrak{Cor}$, then the diagram of tensor functors
$$
\begin{diagram}
\node{} \node{\END_\C(.\So\boxtimes_\D
\N)}\arrow[2]{s,l}{\ad_{\phi\boxtimes_{\D}\N}}
\arrow{se,t}{\ad_\alpha} \\
\node{\END_\D(.\N)} \arrow{ne,t}{\operatorname{Id}_\So\boxtimes_\D(-)}
\arrow{se,b}{\operatorname{Id}_{\So'}\boxtimes_\D(-)}\node{} \node{\END_\C(.\M)}\\
\node{}\node{\END_\C(.\So'\boxtimes_\D \N)}
\arrow{ne,b}{\ad_{\alpha'}}
\end{diagram}
$$
commutes up to a natural tensor isomorphism. Since
$\ad_{\phi\boxtimes_{\D}\N}$ is a tensor equivalence, the tensor
functors associated to $(\So,\alpha)$ and $(\So',\alpha,)$ are
tensor isomorphic, and $\mathcal{K}^{-1}$ is well defined.
\end{proof}
Recall that a functor is an equivalence if and only it is faithful,
full and essentially surjective \cite[Theorem 1,
p.91]{MacLane-book}. Since $\mathcal{K}$ is the identity at the level of objects it is  essentially
surjective. Hence,  by Lemma \ref{Lemma K
is faithful and full} $\mathcal{K}$ is an equivalence of
categories. The  previous discussion gives a proof of Theorem
\ref{description of Funct}.

\section{Tensor equivalences between Morita duals of fusion categories}\label{Seccion equivalencis entre equivariantizaciones}

Let $\D$ be a fusion category and $\N$ be a left $\D$-module
category. Our first goal is to describe, only in terms of data
associated to $\D$ and $\N$, all possible pairs $(\C,\M)\in
\mathfrak{Funct}$ such that $\C_\M^*\cong \D_\N^*$ as fusion
categories. The next proposition gives a necessary and sufficient
condition in terms of the data of $\mathfrak{Cor}$ to say when a
functor between categorical duals is a tensor equivalence.
\begin{proposition}\label{equivalencias en Cor}
A tensor functor $F:\END_{\D}(\N)\to \END_{\C}(\M)$ is an
equivalence if and only if $\So_F = \Fun_{\END_\D(.\N)}(.\N,.\M^F)$ is
an invertible $(\C, \D)$-bimodule category.
\end{proposition}
\begin{proof}
If $F$ is an equivalence and we denote by $F^{*}$ a
quasi-inverse, we have that $\mathcal{K}(F\circ F^{*}) =
\mathcal{K}(\id_{\END_{\C}(\M)})$. The functoriality of
$\mathcal{K}$ implies that $\So_F\boxtimes_\C \So_{F^{-1}}\cong \C$
as bimodule categories. Thus, by \cite[Proposition 4.2]{ENO3},
$\So_F$ is an invertible $(\C, \D)$-bimodule category. Conversely, if $\So_F$ is invertible the functor $L$, defined in the
proof of Lemma \ref{Lemma K is faithful and full}, is an equivalence. 
Then $F \cong \mathcal{K}^{-1}((\So_F,\varepsilon))=\ad_\alpha\circ L^*$ is a tensor equivalence.
\end{proof}
Let $(\D, \N)$ be an object in the category $\mathfrak{Funct}$ and $\So$ be an indecomposable right  $\D$-module category. It is possible associate to the triple $(\D, \N, \So)$  an equivalent object to $(\D, \N)$ in $\mathfrak{Funct}$.
To do so consider the pair $(\END_{\D}(\So),\So \boxtimes_{\D}\N)\in \mathfrak{Funct}$.
It follows from Proposition \ref{equivalencias en
Cor} that $(\So,\id_{\So \boxtimes_{\D}\N})$ is an equivalence in
$\mathfrak{Cor}$. Then, by Theorem \ref{description of Funct}, we
have an equivalence $\D^*_{\N}\cong (\END_{\D}(\So))^*_{\So
\boxtimes_{\D}\N}$ in $\mathfrak{Funct}$ as we desired.
\begin{theorem}\label{Teor aplicacion a contruccion equivalents}
Let $(\D, \N)$ be an object in $\mathfrak{Funct}$. Any other pair
$(\C, \M)$ in $\mathfrak{Funct}$ such that $\D^*_{\N}\cong\C^*_{\M}$
can be obtained from triple $(\D,\N,\So)$ as above. In other words, there exists a
indecomposable right $\D$-module $\So$ and a tensor equivalence
$F:\D^*_{\So}\to \C$ such that $\M^F\cong \So\boxtimes_{\D}\N$ as
left $\D^*_{\So}$-module categories.
\end{theorem}

\begin{proof}Let $(\C, \M)$ be an object in $\mathfrak{Funct}$ such that
$\D^*_{\N}\cong\C^*_{\M}$ via a tensor functor $G$. By Theorem
\ref{description of Funct}, the pair $(\So_G,\varepsilon)$ in
$\mathfrak{Cor}$ is an equivalence from $(\C, \M)$ to $(\D, \N)$. In
view of Proposition \ref{equivalencias en Cor}, $\So_G$ is an invertible
$(\C,\D)$-bimodule. Hence, by \cite[Proposition 4.2]{ENO3}, the
functor of left multiplication $L: \C\to \END_{\D}(\So_G)$ is an
equivalence. In this way, we have an equivalence $F := L^{-1}:
\D^*_{\So_G}\to \C$. By Subsection
\ref{subsection-duality-functors}  $\M^F\cong \So_G\boxtimes_{\D}\N$ as
left $\D^*_{\So_G}$-module categories.
\end{proof}

\subsection{Functors between equivariantizations of fusion categories}\label{Section functor between equivariant fusion cat}

\subsubsection{Equivariant fusion categories}
Let $\M$ be a category (respectively $\C$ a monoidal  category). We
will denote by $\underline{\operatorname{Aut}(\M)}$ (respectively
$\underline{\operatorname{Aut}_\otimes(\C)}$) the monoidal category where
objects are autoequivalences of $\M$ (respectively tensor
autoequivalences of $\C$), arrows are  natural isomorphisms
(respectively tensor natural isomorphisms) and the tensor product is the composition of functors.

An \emph{action} of a finite group $G$ on $\M$ (respectively $\C$) is a monoidal functor  $*:\underline{G}\to
\underline{\operatorname{Aut}(\M)}$ (respectively $*:\underline{G}\to
\underline{\operatorname{Aut}_\otimes(\C)}$),  where $\underline{G}$ denote the discrete monoidal category. Recall that objects in $\underline{G}$ are elements of $G$ and the tensor product is given by the product of $G$.

Let $G$ be a group acting on  $\M$ (respectively  $\C$) with action $*: \underline{G}\to \underline{\operatorname{Aut}(\M)}$ (respectively $*: \underline{G}\to \underline{\operatorname{Aut}_\otimes(\C)})$, thus we have the
following data \begin{itemize}
\item  functors $\sigma_*: \M\to \M$ (respectively, tensor functor $\sigma_*: \C\to \C$), for each $\sigma\in G$,
\item natural isomorphism (respectively, natural tensor isomorphisms) $\phi(\sigma,\tau): (\sigma\tau)_*\to \sigma_*\circ \tau_*$, for all $\sigma, \tau \in G$;
\end{itemize}
satisfying some coherence conditions, see for example \cite[Section 2]{Tam-act}.

Notice that an action of a finite group $G$ on a category $\M$ is exactly the same as a $\vect_G$-module structure over $\M$ (see Appendix \ref{Apendice} for more details).

\begin{ejem}\label{Example action over pointed fusion categories}

Let $G$ and $F$ be finite groups. Given $\omega \in Z^3(F,k^{\times})$, an
action of $G$ on $\vect_F^\omega$ is determined by a homomorphism
$*:G\to \Aut(F)$ and maps
\begin{align*}
\gamma: G\times F\times F&\to k^{\times}\\
\mu:G\times G\times F&\to k^{\times}
\end{align*}such that
\begin{align*}
\frac{\omega(a,b,c)}{\omega(\sigma_*(a) ,\sigma_*(b),\sigma_*( c))}&= \frac{\mu(\sigma;b,c)\mu(\sigma;a,bc)}{\mu(\sigma;ab,c)\mu(\sigma;a,b)},\\
\frac{\mu(\sigma;\tau_*(a) ,\tau_*(b))\mu(\tau;a,b)}{\mu(\sigma\tau;a,b)}&=\frac{\gamma(\sigma,\tau;ab)}{\gamma(\sigma,\tau ;a)\gamma(\sigma,\tau;b)},\\
\gamma(\sigma\tau,\rho;a)\gamma(\sigma,\tau;\rho_*(a) )&=\gamma(\tau,\rho;a)\gamma(\sigma,\tau\rho;a),
\end{align*}for all $a,b,c\in F, \sigma,\tau,\rho\in G$.

The action is defined as follows: for each $\sigma\in G$, the
associated tensor functor $\sigma_*$ is given by $\sigma_*(k_a):=k_{\sigma_*(a)}$, constraint $\psi(\sigma)_{a,b}= \gamma(\sigma;a,b)\id_{k_{ab}}$ and the tensor natural isomorphism is
$$\phi(\sigma,\tau)_{k_a}=\mu(\sigma,\tau;a)\id_{k_a},$$
for each pair
$\sigma,\tau \in G, a\in F$.
\end{ejem}

Given an action $*:\underline{G}\to
\underline{\operatorname{Aut}_\otimes(\C)}$ of a finite group $G$ on $\C$, the
\emph{$G$-equivariantization}  of $\C$ is the category denoted by $\C^G$ and defined as follows. An object in $\C^G$
is a pair $(V, f)$, where $V$ is an object of $\C$ and $f$
is a family of isomorphisms $f_\sigma: \sigma_*(V) \to V$, $\sigma
\in G$, such that
\begin{equation}\label{deltau} f_{\sigma\tau}=
f_\sigma \sigma_*(f_\tau)\phi(\sigma,\tau),\end{equation}
for all $\sigma, \tau \in G$. A $G$-equivariant morphism
$\phi: (V, f) \to (W, g)$ between $G$-equivariant objects
is a morphism $u: V \to W$ in $\mathcal C$ such that
$g_\sigma\circ \sigma_*(u) = u\circ f_\sigma$, for all $\sigma \in
G$.

Note that for the definition of $\C^G$, is not necessary a monoidal structure over $\C$.
If the category $\C$ is a fusion category and the action of $G$ is by tensor
autoequivalences $*: \underline{G}\to \underline{\operatorname{Aut}_\otimes(\C)},$ then we
have  natural isomorphisms
$$\psi(\sigma)_{V,W}:\sigma_*(V)\otimes \sigma_*(W)\to \sigma_*(V\otimes
W),$$ for all $\sigma\in G$, $V, W\in \C$. Thus $\C^G$ is a fusion category with  tensor product defined by
\begin{align*}
(V, f)\otimes (W, g):= (V\otimes W, h),
\end{align*}where $$h_\sigma= u_\sigma v_\sigma\psi(\sigma)_{V,W}^{-1},$$and the unit
object is $(\bf{1}, \operatorname{id}_{\bf{1}})$.
\begin{ejem}\label{Ejem centro de punteada como equivariantizacion}
\begin{enumerate}[leftmargin=*]
\item If $G$ is a finite group acting trivially on $\vect$, then $\vect ^G= \Rep(G)$.
\item Let $G$ be a finite group and  $\omega \in Z^3(G,k^{\times})$. The finite group $G$ acts by conjugation on $\vect_G^\omega$ via the maps
\begin{align*}
\mu(\sigma,\tau; \rho) &:= \frac{\omega(\sigma\tau\sigma^{-1}, \tau,\rho)}
{\omega(\sigma\tau\sigma^{-1}, \sigma\rho\sigma^{-1}, \sigma)\omega(\sigma, \tau, \rho)},\\
\gamma(\sigma;\tau,\rho) &:= \frac{\omega(\sigma, \tau, \rho) \omega(\sigma\tau\rho(\sigma\tau)^{-1},\sigma,\tau)}
{\omega(\sigma, \tau\rho\tau^{-1}, \tau)},
\end{align*}
for all $\sigma$, $\tau$, $\rho\in G$ (see Example \ref{Example
action over pointed fusion categories}). In this case, the equivariantization
$(\vect_G^\omega)^G=\mathcal{Z}(\vect_G^\omega)$ is the Drinfeld'
center of $\vect_G^\omega$ or, equivalently, the category of representations of
the twisted Drinfeld double of $G$, see \cite{DPP}.
\end{enumerate}
\end{ejem}

\subsection{The semi-direct product and their module categories}

Given an action $*:\underline{G}\to
\underline{\operatorname{Aut}_\otimes(\C)}$ of $G$ on a fusion category $\C$, in addition to $\C^G$ another fusion
category associated is the \emph{semi-direct product} fusion
category, denoted by $\C\rtimes G$ and defined as follows. As $k$-linear category $\C\rtimes G= \bigoplus_{\sigma\in G}\C_\sigma$,
where $\C_\sigma =\C$. The tensor product is given by
$$[X, \sigma]\otimes [Y,\tau]:= [X\otimes \sigma_*(Y), \sigma\tau],\
\  \   X,Y\in \C,\  \   \sigma,\tau\in G,$$ and the unit object is
$[\bf{1},e]$. See \cite{Tam-act} for the associativity constraint.
\medbreak
In order to understand the module categories over $\C\rtimes G$ it is useful  to recall the notion of equivariant module categories \cite{ENO2}. We will use
the approach given in \cite{Ga1}.

Let $G$ be a group and $\C$ be a tensor category equipped with an action $*$ of $G$.
Let $\M$ be a module category over $\C$. 
We define the $G$-graded monoidal category $\underline{\Aut^G_{\C}(\M)}$ of $G$-invariant autoequivalences of $\M$ in the following way. The objects are pairs $(\sigma,(T,c))$, where  $\sigma\in G$, and $(T,c):\M\to\M^{\sigma_*}$ is a $\C$-module equivalence. If $(\sigma, (T,c)), (\tau,(U,d) )\in \underline{\Aut^G_{\C}(\M)} $,  then $(\sigma\tau, (T\circ
U,b))\in \underline{\Aut^G_{\C}(\M)}$, where
\begin{align}\label{invariant-composition} b_{X,M}=((\gamma_{\sigma,\tau})_X\ot\id_{T\circ U(M)})c_{\tau_*(X),U(M)}T(d_{X,M}),
\end{align}
for all $X\in \C$, $M\in \M$. The arrows of $\underline{\Aut^G_{\C}(\M)}$ are just natural isomorphisms of $\C$-module categories.

A \emph{$G$-equivariant $\C$-module category}
is a $\C$-module category $\M$ equipped with a $G$-graded monoidal
functor $(\Phi,\mu):\underline{G}\to \underline{\Aut^G_{\C}(\M)}$.

\begin{ejem}\label{Example C rtimes G acting on C} Let $\C$ be a fusion category with an action of $G$ given by the data $\{(\sigma_*,\psi(\sigma),\phi(\sigma,\tau) )\}_{\sigma,\tau \in G}$. Then $\C$  is a $G$-equivariant module category over itself, where $\Phi(\sigma)=(\sigma_*, \psi(\sigma))$ and
$\mu_{\sigma,\tau}=\phi(\sigma,\tau)$, for all $\sigma, \tau\in
G$.
\end{ejem}

Given a $G$-equivariant $\C$-module category $\So$ we define the fusion
category $\END^G_\C(\So)$ of
$G$-equivariant $\C$-endofunctor of $\So$  as follows. Objects are
pairs $(L,\eta)$, where $L:\So\to \So$ is a $\C$-module endofunctor
and $\eta(\sigma):L\circ \sigma_*\to \sigma_*\circ L$ are natural isomorphisms such that
$$\phi(\sigma,\tau)_{L(X)}\circ \eta(\sigma\tau)_X=
\sigma_*(\eta(\tau)_X)\circ \eta(\sigma)_{\tau_*(X)}\circ
L(\phi(\sigma,\tau)_{X}),$$
for all $\sigma,\tau \in G, X\in \C$.
The arrows and composition in $\END^G_\C(\So)$ are defined in the
obvious way.
\begin{remark}\label{remark invariant}
\begin{enumerate}[leftmargin=*]
\item A $G$-equivariant module category over $\C$ is the same as a $\C\rtimes G$-module category, \cite[Proposition 5.12]{Ga1}. Explicitly, if $\So$ is a
$G$-equivariant $\C$-module category then $\So$ is a $\C\rtimes
G$-module, with action $\otimes: (\C\rtimes G)\times \So\to \So$ given
by $[X,\sigma]\otimes M=X\otimes \sigma_*(M)$, for all $X\in \C$, $\sigma\in G$ and
$M\in \So$.

\item The fusion category  $\END^G_\C(\So)$ is canonically equivalent to $\END_{\C\rtimes G}(\So)$.
\item Since $\So^G\cong\Fun_{\C\rtimes G}(\C,S)$ then $\So^G$ has a canonical structure of
$(\C^G$,$\END^G_\C(\So))$-bimodule category.
\end{enumerate}
\end{remark}

Applying Theorem \ref{Teor aplicacion a contruccion equivalents} to the description of the Drinfeld's center of a
pointed fusion category as an equivariantization (see Example
\ref{Ejem centro de punteada como equivariantizacion}) we have the
following result that can be seen as a generalization of
\cite[Corollary 1.5]{Naidu-Nik}.
\begin{corollary}\label{Corol centro gt}
Let $G$ and $H$ be finite groups and $\omega_G \in H^3(G,k^{\times}),
\omega_H \in H^3(H,k^{\times})$. Then $\mathcal{Z}(\vect_G^{\omega_G})\cong
\mathcal{Z}(\vect_H^{\omega_H})$ as tensor categories (not
necessarily as braided categories) if and only if there is a
pointed $G$-equivariant $\vect_{G}^{\omega_G}$-module category 
$\M$ and a tensor equivalence $\Phi:\vect_{H}^{\omega_H}\rtimes H \to
\END_{\vect_G^{\omega_G}}^G(\M)$ such that $\vect_{H}^{\omega_H} \cong \big (\M^G \big )^\Phi$ as $H$-equivariant
$\vect_{H}^{\omega_H}$-module categories.
\end{corollary}

\subsection{Tensor equivalent equivariant fusion categories}
Recall that the category $\C$ has a canonical $\C\rtimes G$-module
structure, by Example \ref{Example C rtimes G acting on C} and Remark \ref{remark invariant}. The categories $\C\rtimes G$ and
$\C^G$ are Morita equivalent since $(\C\rtimes
G)^*_{\C}\cong \C^G$ by \cite[Proposition 3.2]{Nik}.

Combining Theorem
\ref{description of Funct} and \cite[Proposition 3.2]{Nik} we get:
\begin{corollary}\label{description of functors entre equivariantizations}
Tensor functors between equivariantizations of fusion categories
under the action of a finite group are in correspondence with the arrows of the subcategory of $\mathfrak{Cor}$ whose objects are of the form $(\C\rtimes G,\C)$, where $G$ is a finite group acting on a fusion category $\C$.
\end{corollary}

Let $G$ be a finite  group and $\C$ be a fusion category. We will
say that  $\C$ is  $G$-graded if there is a decomposition
$\C=\oplus_{x\in G}\C_x$ of $\C$ into a direct sum of full abelian
subcategories such that the bifunctor $\otimes$ maps
$\C_x\times \C_y$ to  $\C_{xy}$, for all $x, y\in
G$. See \cite{ENO} for more details.

Before presenting the main result of this section, we need a
technical lemma.
\begin{lemma}{\cite[Corollary 6.4]{Ga3}}\label{lemma para equivariantes equivalentes}
Let $G$ be a finite group and  $\C=\bigoplus_{x\in G}\C_x$
be a $G$-graded fusion category. Let $\M$ be an indecomposable
$\C$-module category which remains indecomposable as a $\C_e$-module category. Then  $\END_{\C}(\M)\cong
\END_{\C_e}(\M)^{G}$, that is, $\END_{\C}(\M)$ is a
$G$-equivariantization of $\END_{\C_e}(\M)$.
\end{lemma}
\begin{theorem}\label{teorema Equivalent Equivariantizations}
Let $\C$ be a fusion category, $G$ and $H$ be finite groups and $*:\underline{G}\to \underline{\Aut_\otimes(\C)}$ be an
action of $G$ on $\C$.
\begin{enumerate}
\item Let $\So$ be an indecomposable $G$-equivariant $\C$-module category and  let $\END^G_\C(\So)=\bigoplus_{h\in H} \END^G_\C(\So)_h$ be a faithfully $H$-grading such that $\So^G$  is equivalent to
$\END^G_\C(\So)_e$ as $\END^G_\C(\So)_e$-module categories. Then $\END^G_\C(\So)_{\So^G}^*\cong \big (\END^G_\C(\So)_e \big )^H$, that is,
$\END^G_\C(\So)_{\So^G}^*$ is an $H$-equivariantization,
and $\big (\END^G_\C(\So)_e \big )^H \cong \C^G$ as fusion categories.
\item Conversely, for every fusion category of the form $\D^H$ that is tensor equivalent to $\C^G$
there exists a $G$-equivariant $\C$-module category $\So$ and a faithful $H$-grading
in $\END^G_\C(\So)$ such that $\END^G_\C(\So)_e\cong \D$ and
$\D^H\cong \big (\END^G_\C(\So)_e \big )^H$.
\end{enumerate}
\end{theorem}
\begin{proof}
Since $\So^G$ is an invertible  $(\C^G,\END^G_\C(\So))$-bimodule
category the left action defines a tensor equivalence $L: \C^G\to
\END_{\END^G_\C(\So)}(\So^G.)$, by \cite[Proposition 4.2]{ENO3}. Then, Lemma
\ref{lemma para equivariantes equivalentes} implies that
$\END_{\END^G_\C(\So)}(\So^G.)$ is an $H$-equivariantization of
$\Big( \END(\So^G)_{\END^G_\C(\So)}\Big )_e$.

Conversely, let $G$ and $H$ be finite groups acting on fusion
categories $\C$ and $\D$ respectively. Since $\C^G\cong
(\C\rtimes G)_\C^*$ and $\D^H\cong (\D\rtimes H)_\D^*$, if $\C^G$ and
$\D^H$ are tensor equivalent. Then by Theorem \ref{description of Funct}, there is
an invertible $(\D\rtimes H,\C\rtimes G)$-bimodule category $\So^{\dag}$
such that $\So^{\dag} \boxtimes_{\C\rtimes G} \C\cong \D$ as  $\D\rtimes H$-module
categories. The category $\C$ is an invertible  $(\C\rtimes G,\C^G)$-bimodule category, by \cite[Theorem 4.1]{Tam-act}. If we define  $\So:=(\So^{\dag})^{\op}$ as a 
$(\C\rtimes G$,$\D\rtimes H)$-bimodule category, then the equivalences
\begin{align*}
\D &\cong \So^{\dag} \boxtimes_{\C\rtimes G} \C \\
&\cong \Fun_{\C\rtimes G}(.\So,.\C)\\
&\cong \Fun_{\C\rtimes G}(.\C,.\So)^{\op}\\
& \cong (\So^G)^{\op}
\end{align*}
imply that $\So^G$ is an invertible $(\C^G,\D\rtimes H)$-bimodule
category with $\So^G\cong \D^{\op}$ as right $\D\rtimes H$-module categories.
Using the bimodule category structure, we have a tensor
equivalence $R:(\D\rtimes H)^{\operatorname{op}}\to (\C^G)^*_{S^G}\cong_\otimes \END^G_\C(\So)$.
Then $\END^G_\C(\So)$ has an $H$-grading and  $\So^G\cong \END^G_\D(\So)_e$ as left
$\END^G_\D(\So)_e$-module categories.
\end{proof}

\subsubsection{Example: Isocategorical groups}
Two finite groups $G$ and $H$ are called \emph{isocategorical} if their categories of representations are tensor equivalent \cite{EG}.

Using Theorem \ref{Teor aplicacion a contruccion equivalents} we can give
an alternative proof to \cite[Corollary 6.2] {Davy}, a key result in the classification of isocategorical groups,  see \cite{EG, Davy, IH}. We want to draw the attention of the reader at this point, the classification of isocategorical groups given in \cite{EG} and \cite{Davy} is more explicit. 

In this subsection we will follow the notation of the Appendix \ref{Apendice}. 
Let $G$ be a finite group. Consider $G$ acting trivially on $\vect$,
then $\Rep G = (\vect)^G$ . We will apply Theorem \ref{Teor aplicacion a contruccion equivalents} to
 $\D=\vect$. Let $\M$ be a $\vect_G$-module category.
Since extensions of $\vect$ are pointed fusion categories, the $\vect_G$-module category $\M$ must be pointed (see Appendix \ref{Subseccion module cat}) with $\M^G\cong
\vect$. Thus $\M^G$ must have rank one.
By \cite[Theorem 3.4]{Deepak} (see also Proposition \ref{prop clasificacion pointed module}, pointed module categories over $\vect_G$ are in
correspondence with pairs $(A,\psi)$, where $A$ is a normal abelian
subgroup of $G$ and $\psi\in H^2(A,\psi)$ is $\Ad_G$-invariant. Since
$\M(X,\mu)^G= \Fun_{\vect_G}(\vect_G,\M(X,\mu))$, it follows from \cite[Proposition 3.1]{ostrik} (see also Corollary \ref{corolario alternativo teo de ostrik}) that simple
objects are in correspondence with $\psi$-projective representations of $A$. In particular, the rank one condition on $\M(X,\mu)^G$ is equivalent to the non-degeneracy of $\psi$, that is, $k_\psi[A]$ is a simple algebra.

By Theorem \ref{Teor aplicacion a contruccion equivalents},
every group $H$ such that $\Rep(H)\cong \Rep(G)$ as fusion categories
can be constructed as $\Aut_{\vect_G}(\M(A,\psi))$
(see Appendix \ref{Subseccion module cat}),
where $A$ is a normal abelian subgroup of $G$ and $\psi\in H^2(A,k^{\times})$ is a
non-degenerate $\Ad_G$-invariant cohomology class.

\section{Applications to the Brauer-Picard group}\label{Seccion aplicaciones BrPic}

\subsection{Proof of Theorem \ref{RZ}}\label{Proof of Theorem RZ}

By Theorem \ref{description of Funct} we have the following exact sequences of groups:
\begin{align*}
& \xymatrix{  &\Aut_{\C}(\M)  \ar[r]^{\operatorname{conj}_{\M}} \ar[dd]^{\Omega} &
\Aut_{\otimes}(\C_{\M}^*)\ar[r]^{\pi_1\circ \mathcal{K}}\ar[dd]^{\simeq \text{ }
\mathcal{K}} & \operatorname{BrPic}(\C)\ar[dd]^{=}
\\
\\1 \ar[r]^{} &
\operatorname{I}(\C, \M)  \ar[r]  & \Aut_{\mathfrak{Cor}}(\C, \M)
\ar[r]^{\pi_1} & \operatorname{BrPic}(\C),
}
\end{align*}
where  $\operatorname{I}(\C, \M)=\{[(\So,\alpha)]\in \End_\mathfrak{Cor}((\C,\M)): \So\cong\C
\text{ as $\C$-bimodules }\}$ and
$\operatorname{conj}_\M$ is conjugation. The map $\Omega$ is defined by $\Omega(F)= (\C, \operatorname{Id}_\C\boxtimes_\C F)$ and
$(\pi_1\circ \mathcal{K})(G)=\So_G = \Fun_{\END_\C(.\M)}(.\M,.\M^G)$
is the invertible $\C$-bimodule category associated by $
\mathcal{K}$ to $G\in \Aut_\otimes(\END_{\C}(\M))$.

Let us consider the abelian
group of (isomorphism classes of) invertible objects
$\Inv(\mathcal{Z}(\C))$ of the Drinfeld center $\mathcal{Z}(\C)$ of
$\C$. For every $\C$-module category $\M$ we have a group
homomorphism
\begin{align*}
s:\Inv(\mathcal{Z}(\C))\to \Aut_\C(\M),
\end{align*}
where $s_X(M):= X\otimes M$ and $\gamma_{V,M}: s_X(V\otimes M)\to
V\otimes s_X(M)$ is given by $\gamma_{V,M}:=c_{X,V}\otimes
\id_{M}$, for all $(X,c_{X,-})\in \Inv(\mathcal{Z}(\C))$, $V\in \C$,
$M\in \M$. From here we obtain the sequence
\begin{equation}\label{exact seq}
    1\to\ker(s)\to \Inv(\mathcal{Z}(\C))\overset{s}{\to} \Aut_{\C}(\M)\to\Aut_{\otimes}(\C_{\M}^*)\to 
\operatorname{BrPic}(\C).
\end{equation}

Since $\Omega$ is surjective, if we prove that
$\ker(\Omega)=\operatorname{Im}(s)$ then
$\operatorname{Im} (\operatorname{conj}_\M)=\operatorname{ker}(\pi_1\circ \mathcal{K})\cong \operatorname{I}(\C, \M)$. It follows that the sequence \eqref{exact seq} is exact. We verify that  $\ker(\Omega)=\operatorname{Im}(s)$ as follows. If $\Omega(F)= \id_{(\C,\M)}$,  there is an invertible $\C$-bimodule
functor $\phi:\C\to \C$ such that $\phi\boxtimes_\C\M $ is
isomorphic to $\operatorname{Id}_\C\boxtimes_\C F\cong F$ as $\C$-module
functors. Since, every invertible $\C$-bimodule functor has the form
$X\otimes(-)$ for a unique $X\in \Inv(\mathcal{Z}(\C))$, then
$\phi\boxtimes_\C\M \cong s_X$, and  $F\in \ker(\Omega)$ if and only
if there is $X\in \Inv(\mathcal{Z}(\C))$ such that $F\cong s_X$,
that is, $\ker(\Omega)=\operatorname{Im}(s)$. \qed

\begin{corollary}\label{corolario RZ caso trenzado}Let $\C$ be a braided fusion category. Set $\M
\cong  \C$. In this case, we have an inclusion of groups 
$\Aut_{\otimes}(\C)\hookrightarrow \operatorname{BrPic}(\C)$.
\end{corollary}
\begin{proof}Since $\C$ is braided, the map $\operatorname{conj}_M$ is trivial. Consequently, the map $\pi_1\circ \mathcal{K}$ is injective.
\end{proof}
Next recall that there are two
equi\-valent realizations of $\Out_{\otimes}(\C)$, one of the presentations
is obtained by considering  $\Out_{\otimes}(\C)$ as the subgroup of $\operatorname{BrPic}(\C)$
whose elements are equivalence classes of \emph{quasi-trivial $\C$-bimodule categories},
that is $\C$-bimodules categories equivalent to $\C$ as right $\C$-module categories
\cite[Subsection 4.3]{ENO3}. The other realization is given by considering $\Out_{\otimes}(\C)$
as the group of equivalence classes of tensor autoequivalences of $\C$ up to pseudonatural isomorphisms \cite[Subsection 3.1]{Ga2}.

\begin{remark}\label{RZ-Inn}
The Rosenberg-Zelinsky sequence in the case that $\M= \C$ has
the form $$1\to \operatorname{Inn}(\C)\hookrightarrow
\Aut_\otimes(\C)\overset{\pi}{\to}  \operatorname{BrPic}(\C),$$ where
$\pi(\sigma)=\C^\sigma$ is the quasi-trivial
$\C$-bimodule that is $\C$ as a right $\C$-module, but  with left
action given by the left multiplication twisted by the tensor
autoequivalence $\sigma$ of $\C$.   The image of $\pi$ is exactly the group
$\Out_\otimes(\C)$ described above. Indeed, for this particular case, the exact sequence of Theorem \ref{RZ} can be rewritten as follows:
$$1\to\Aut_{\otimes}(\id_{\C})\to \Inv(\mathcal{Z}(\C))\overset{\pi_1}{\to} \Inv(\C) \overset{\operatorname{conj}_{\C}}{\to} \Aut_\otimes
(\C)\overset{\pi}{\to} \operatorname{BrPic}(\C),$$
here $\Inv(\mathcal{Z}(\C))\overset{\pi_1}{\to} \Inv(\C)$ corresponds to forgetting the half braiding, that is $\pi_1(X,c_{-,X})=X$. Note that half braidings on $\bf 1$ are exactly monoidal natural isomorphism of the identity functor, hence $\ker(\pi_1)=\Aut_{\otimes}(\id_{\C})$. 
\end{remark} 

\begin{ejem} Let $G$ be a finite group.  Given $f \in Z^n(G, k^{\times})$ and
$\theta\in \Aut(G)$, we will denote by $f^{\theta}$ the
$n$-cocycle in $G$ defined by $f^{\theta}(\sigma_1, \cdots \sigma_n) = f(\theta(\sigma_1), \cdots \theta(\sigma_n))$,
for $\sigma_1,\ldots, \sigma_n \in G$.
This defines an action of $\Aut(G)$ on $H^*(G, k^{\times})$ which factors through $\Inn(G)$ giving rise to an
action of $\Out(G)$.

If $\omega \in Z^3(G,k^{\times})$,  $$\Aut_{\otimes} (\vect_G^{\omega}) =
\{ (f, \gamma_f) \in\Aut(G)\times C^2(G, k^{\times}) :
\delta(\gamma_f) = \frac{\omega^f}{\omega}\}/\sim,$$ where
$(f,\gamma_f)\sim (g,\gamma_g)$ if and only if $f=g$ and there is
$\theta:G\to k^{\times}$ such that
$\delta(\theta)=\frac{\gamma_f}{\gamma_g}$. Thus we have the exact
sequence (see  \cite[Appendix]{ENO3})
\begin{align}
1\to H^2(G, k^{\times}) \to \Aut_{\otimes}(\vect_G^{\omega}) &\to  \Stab_{\Aut(G)} ([\omega])\to 1 \label{Aut vec_G}\\
(f,\gamma_f) &\mapsto f.\notag
\end{align}

Note that the exact sequence \eqref{Aut vec_G} splits if
there is a $\Stab_{\Aut(G)} ([\omega])$-invariant representative
3-cocycle of $[\omega]$. In particular if $\omega=1$,
$\Aut_{\otimes}(\vect_G^{\omega})= H^2(G, k^{\times})\rtimes
\Aut(G)$.

By Example \ref{Ejem centro de punteada como equivariantizacion},
$\Z(\vect_G^{\omega})$ is a $G$-equivariantization with respect to the
$G$-action by conjugation and maps $\gamma $ and $\mu$. A simple object $V$ in $\Z(\vect_G^{\omega})$ is
invertible if and only if $\dim_k(V)=1$. In which case,  $V=k_\rho$, with $\rho
\in \mathcal{Z}(G)$ and $\gamma(-,-;\rho)\in B^2(G,k^{\times})$. Thus we
have an exact sequence  $$1\to \widehat{G}\to
\Inv(\Z(\vect_G^{\omega}))\to \mathcal{Z}(G)^\omega\to 1,$$ where
$\mathcal{Z}(G)^{\omega} = \{\rho\in \mathcal{Z}(G)|
\gamma(-,-;\rho)\in B^2(G,k^{\times})\}.$

The exact sequence of Remark \ref{RZ-Inn} implies that
$\Inn(\vect_G^{\omega}) \simeq G/Z(G)^{\omega}$. Therefore the
second exact sequence of Remark \ref{RZ-Inn} can be rewritten as:
\begin{equation}\label{equation out vecG}
1\to Z(G)^{\omega}\to G \overset{\overline{\Ad}}{\to}\Aut_{\otimes}(\vect_G^{\omega})\to \operatorname{Out}_\otimes(\vect_G^{\omega}).
\end{equation}
\begin{remark}
\begin{enumerate}
\item In general, $Z(G)_{\omega} \varsubsetneq Z(G)$.
For example, when $G$ is an abelian group and
$\mathcal{Z}(\vect_G^{\omega})$ is not pointed.
\item If $\omega=1$, then the exact sequence \eqref{equation out vecG} implies that $\operatorname{Out}_\otimes(\vect_G)=H^2(G,k^{\times})\rtimes \Out(G)$.
\end{enumerate}
\end{remark}
\end{ejem}

\begin{ejem}\label{ejemplo Out de TY}
Let $\C = \TY (A, \chi, \tau)$ be the Tambara-Yamagami category
associated to a finite (necessarily abelian) group $A$, a symmetric
non-degene\-rate bicharacter $\chi : A\times A \rightarrow k^{\times}$
and an element $\tau\in k$ satisfying $|A|\tau^2 = 1$, see \cite{TY}.
Since by \cite[Propositon 1]{Tambara}
$$\Aut_\otimes(\C) = \{\sigma\in \Aut(A) :
\chi(\sigma(a),\sigma(b))= \chi(a,b),\text{ } \forall a, b\in A\},$$
thus
$\operatorname{Inn}(\C)$ is trivial. Then it follows from Remark
\ref{RZ-Inn} that
$$\Out_{\otimes}(\TY (A, \chi,
\tau)) = \Aut_\otimes(\TY (A, \chi, \tau)) \hookrightarrow \operatorname{BrPic}(\TY (A, \chi, \tau)).$$
\end{ejem}

\subsection{Proof of Theorem \ref{coro doble cosets}} \label{Subseccion sobre T(C) y el grupo de BrPic}


Recall from Subsection  \ref{Proof of Theorem RZ}
that $\Out_\otimes(\C)$ is the subgroup of $\operatorname{BrPic}(\C)$ consisting of all quasi-trivial $\C$-bimodule categories. 

We define $\mathcal{T}(\C)$ as the set of equivalence classes of right
$\C$-module categories $\M$ such that $\C\cong \C_\M^*$ as tensor
categories.

By \cite[Proposition 4.2]{ENO3}, we have a map $$\operatorname{BrPic}(\C)\to
\mathcal{T}(\C),$$ given by forgetting the left $\C$-module structure.
This map factorizes through the left action of the subgroup
$\Out_\otimes(\C)$, thus we have a map
$$U: \Out_\otimes(\C) \setminus \operatorname{BrPic}(\C) \to \mathcal{T}(\C).$$
Note that the sets $ \Out_\otimes(\C) \setminus \operatorname{BrPic}(\C)$
and $\mathcal{T}(\C)$ are right  $\operatorname{BrPic}(\C)$-sets in a
natural way. Since $U$ is a map of transitive $\operatorname{BrPic}(\C)$-sets and  the stabilizer of $\C\in \mathcal{T}(\C)$ is $\Out_\otimes(\C)$, the map $U$ is bijective. Theorem \ref{coro doble cosets} is a direct consequence
of the bijectivity of $U$. \qed

\subsubsection{Generalized crossed product for groups}\label{Generalized crossed product}
Let $G$ be a finite group and $F\subseteq G$ be a subgroup. Once a set $Q$ of
simultaneous representatives of the left and right cosets of $F$ in
$G$ is fixed, the group $G$ can be described as a generalized crossed
product as follows. The uniqueness of the factorization $G = FQ$
implies that there are well defined maps $$\fde: Q \times F \to F,
\qquad  \fiz: Q \times F \to Q,$$
$$.: Q \times Q \to Q, \qquad \theta: Q \times Q \to F,$$ determined by the conditions
\begin{align*}qx & = (q \fde x) (q \fiz x), \qquad q \in Q, \, x \in F;\\
pq & = \theta(p, q) p.q, \qquad p, q \in Q.\end{align*} The set
$F\times Q$ with the product $$(u,s)(v,t)=(u(s\fde v)\theta(s\fiz
v,t),(s\fiz v)\cdot t)$$ is a group that we will denote by $F\#^{\fde,\fiz}_{\theta, \ \cdot}Q$.
Moreover, $F\#^{\fde,\fiz}_{\theta,\ \cdot}Q$ is isomorphic to $G$,  \cite[Proposition 2.4]{Beggs}.
\begin{remark}
\begin{enumerate}[leftmargin=*]
\item Let $G$ be a finite group and $F\subset G$ be a subgroup.
Let us recall how to construct a set of simultaneous representatives
of the left and right cosets. First, we fix a set of
representatives $Q$ of the double cosets of $F$ in $G$. For $x\in Q$, let  $\{s_j: j\in J_x\}$ be a set of representatives of the left
cosets of $F\cap xFx^{-1}$ and $\{t_j| j\in J_x\}$ be a set of
representatives of the right cosets of $F\cap x^{-1} Fx$ in $F$.
Notice that $FxF=\cup_{i\in J_x} Fxt_i= \cup_{i\in J_x} s_ixF$. Then
$\{s_jxt_j| x\in Q, j\in J_x \}$ is simultaneously a set of
representatives for the right and left cosets.
\item Theorem \ref{coro doble cosets} and the previous remark provide
a systematic way to reduce the calculations of the Brauer-Picard
group of a fusion category to computations of $\Out_\otimes(\C)$ and
the extra data $\theta, \fiz, \fde$ using only a set of representatives of $\mathcal{T}(\C)/\Out_\otimes(\C)$.
\end{enumerate}
\end{remark}

\begin{ejem}
\begin{enumerate}[leftmargin=*]

\item A finite group is called semisimple if its solvable radical is trivial; equivalently the group has no non-trivial abelian normal
subgroups. Let $G$ be a finite
semisimple group (\textit{e.g.} symmetric groups $\mathbb{S}_n$ $(n>4)$ or non abelian
simple groups) and $\omega \in H^3(G,k^{\times})$. Since every module
category in $\mathcal{T}(\vect_{G}^\omega)$ is pointed 
$\mathcal{T}(\vect_{G}^\omega)=\{\vect_{G}^\omega\}$, so
$\operatorname{BrPic}(\vect_{G}^\omega)=\Out_\otimes(\vect_{G}^\omega)$.
\item Let $p$ be a prime number. For any $0\neq \omega \in H^3(\mathbb{Z}/p\mathbb{Z},k^{\times})\cong \mathbb{Z}/p\mathbb{Z}$, we have that 
$\mathcal{T}(\vect_{\mathbb{Z}/p\mathbb{Z}}^\omega)=\{\vect_{\mathbb{Z}/p\mathbb{Z}}^\omega\}.$ 
Thus $$\operatorname{BrPic}(\vect_{\mathbb{Z}/p\mathbb{Z}}^\omega)=\Out_\otimes(\vect_{\mathbb{Z}/p\mathbb{Z}}^\omega)\\
\cong \operatorname{Stab}_{\Aut(\mathbb{Z}/p\mathbb{Z})}([\omega])=\{\text{id}_{\mathbb{Z}/p\mathbb{Z}}\}.$$
\item Let $\C = \TY (\mathbb{Z}/p\mathbb{Z}, \chi, \tau)$
be a non group-theoretical Tambara-Yamagami category, that is,
$\chi(1,1)=e^{\frac{2\pi k }{p}}$ where $k\in
\mathbb{Z}/p\mathbb{Z}$ is a quadratic non-residue. It follows by
\cite[Proposition 5.7]{Ga3} that the only indecomposable $\C$-module
category is $\C$ itself. Thus, $\mathcal{T}(\C)=\{\TY
(\mathbb{Z}/p\mathbb{Z}, \chi, \tau)\}$ and hence $\operatorname{BrPic}(\TY
(\mathbb{Z}/p\mathbb{Z}, \chi, \tau))=\Out_\otimes(\TY
(\mathbb{Z}/p\mathbb{Z} \chi, \tau))$. Then, by Example \ref{ejemplo
Out de TY} we have that
$$\operatorname{BrPic}(\TY (\mathbb{Z}/p\mathbb{Z}, \chi, \tau))\cong\mathbb{Z}/2\mathbb{Z}.$$
\end{enumerate}
\end{ejem}

\section{Invertible bimodule categories over pointed fusion categories and their tensor product }\label{Seccion bimodules de pointed fusion cat}
The goal of this section is to describe explicitly bimodule
categories over pointed fusion categories and their tensor product
in order to provide all ingredients for applying Theorem
\ref{description of Funct} to concrete examples of group-theoretical
fusion categories. 

A \emph{group-theoretical} fusion category is, by definition, a
fusion category Morita equivalent to a pointed fusion category
$\vect_G^{\omega}$. See Appendix \ref{Apendice} for more details.

The following corollary of Theorem \ref{description of Funct} provides an
\textit{implicit} answer to Problem 10.1 posted by Gelaki in
http://aimpl.org/fusioncat/10/.
\begin{corollary}\label{description of funct entre GT}
Tensor functors between group-theoretical fusion categories are in
correspondence with the arrows of the subcategory of
$\mathfrak{Cor}$ whose objects are of the form $(\C,\M)$, with $\C$
a pointed fusion category.
\end{corollary}

\subsection{Goursat's Lemma and bitransitive bisets}

Let $G_1,$  $G_2$ be groups and $X$ be a $G_1$-$G_2$-biset. We can regard any $G_1$-$G_2$-biset $X$ as a left $G_1 \times
G_2$-set equipped with the action  $(h, k) \cdot x = h \cdot
x \cdot k^{-1}$, for all $x \in X$, $h\in G_1, k \in G_2$.
Reciprocally, any $G_1 \times G_2$-set can be regarded as a
$G_1$-$G_2$-biset. It is easy to see that this  defines an
equivalence between the categories of $G_1$-$G_2$-bisets and left
$G_1\times G_2$-sets.

A $G_1$-$G_2$-biset is called \emph{transitive} if for some (and so
for every) $x\in X$, $G_1\cdot x\cdot G_2=X$. A $G_1$-$G_2$-biset
$X$ is transitive if and only if $X$ is transitive as  $G_1\times
G_2$-set. Thus, isomorphism classes of transitive $G_1$-$G_2$-bisets
are classified by conjugacy classes of subgroups of $G_1\times G_2$.

Next, we recall the description of the subgroups of a direct product
of groups known as Goursat's Lemma. The proof is a simple exercise in group theory, see \cite[Exercise 5, p. 75]{Lang}.

\begin{lemma}
Let $H$ be a subgroup of
$G_1\times G_2$. Define
\begin{align*}
H_1 &=\{a\in G_1 | (a,b)\in   H, \text{ for some } b  \in G_2\}\\
H_2 &=\{b\in G_2 | (a,b)\in H, \text{ for some } a  \in G_1\}\\
H_1^2=\{a\in G_1 | & (a,1)\in H\},\  \ H_2^1=\{b\in G_2 | (1,b)\in
H\}.
\end{align*}
Then $H_1^2\unlhd H_1 $ and $H_2^1\unlhd H_2 $ are normal subgroups. The map $f_H:
H_1/H_1^2\to H_2/H_2^1$ given by $f_H(aH_1^2)=bH_2^1$ is an
isomorphism.

Conversely, every subgroup $H\subset G_1\times G_2$ is
constructed as a fiber product in the following way: let
$H_i^j\unlhd H_i\subset G_i$ be subgroups and $f_H: H_1/H_1^2\to
H_2/H_2^1$ an isomorphism. Then $H=H_1\times_{f_H}
H_2=\{(h_1,h_2)| f_H(aH_1^2)=bH_2^1\}\subset
G_1\times G_2$.
\end{lemma}

Let $G_1$ and $G_2$ be  groups and $X$ a $G_1$-$G_2$-biset. We will
say that $X$ is a \textit{bitransitive biset} if $X$ is transitive
as both a right $G_2$-set and a left $G_1$-set. A subgroup $H\subset
G_1\times G_2$ is called a \textit{bitransitive subgroup} if
$(G_1\times G_2)/H$ is a bitransitive $G_1$-$G_2$-biset.

Obviously every bitransitive $G_1$-$G_2$-biset is transitive as
$G_1\times G_2$-set but the conversely is not true, e.g., $G_1$ with trivial $G_2$-action and the regular $G_1$ is a transitive $G_1\times G_2$-set but not bitransitive.

Let $X$ be a $G_1$-$G_2$-biset and $x\in X$. We define the left,
right, and bi-stabilizer subgroups of $x$ as $\Stab_r(x)=\{g\in G_2|
xg=x\}$, $\Stab_l(x)=\{g\in G_1| gx=x\}$, and  $\Stab_{bi}(x)=\{(h,k)\in
G_1\times G_2| hxk=x\}$, respectively.  Notice that if $H=\Stab_{bi}(x)$
then, in the previous notation, $\Stab_l(x)=H_1^2$ and $\Stab_r(x)= H_1^2$.
\begin{remark}\label{observacion ordenes de estabilizadores}
If $X$ is a bitransitive $G_1$-$G_2$-set then:
\begin{enumerate}[leftmargin=*]
\item  $|X|=|G_2|/|\Stab_r(x)|=|G_1|/|\Stab_l(x)|= |G_1||G_2|/|\Stab_{bi}(x)|$. In particular, $|\Stab_r(x)|=|\Stab_l(x)|$ when $|G_1|=|G_2|$.
\item If $\Stab_r(x)=1$ (or  $\Stab_l(x)=1$) then $H_i=G_i$. Moreover, there is a unique group
isomorphism $f:G_1\to G_2$ such that $\Stab_{bi}(x)=\{(g,f(g))| g\in
G_1\}$. This kind of bisets are called \emph{bitorsors}.
\end{enumerate}
\end{remark}
Let $X$ be a right transitive $G_1$-$G_2$-biset and $x\in X$. Set
$H=\Stab_r(x)$. We can, and  will, assume that $X=H\backslash G_2$ as a
 right $G_2$-set.  Notice that every $g\in G_1$ defines a map
$\widehat{g}:X\to X, y\mapsto gy$ that is an automorphism of right
$G_2$-sets. The left action is determined by the map
$\widehat{(-)}: G_1\to \Aut_{G_2}(X), g\mapsto \widehat{g}$. Since we
are assuming that $X=H\backslash G_2$ as right $G_2$-set, then
$\Aut_{G_2}(X)\cong \text{N}_{G_2}(H)/H$. Thus, the map
$\widehat{(-)}$ defines, and is defined by, a group morphism
$\pi:G_1\to \text{N}_{G_2}(H)/H$. 
\begin{proposition}
Let $G_1$, $G_2$ be groups and  $X$ be a bitransitive $G_1$-$G_2$-biset. For any $x\in X$, the subgroups
$\Stab_r(x)\unlhd G_2$ and $\Stab_l(x) \unlhd G_1$ are normal and the group homomorphism $\pi$ induces a group isomorphism $\tilde{\pi} :G_1/\Stab_l(x)\to G_2/\Stab_r(x)$.

Conversely, a pair of normal subgroups $N_1 \unlhd G_1, N_2\unlhd
G_2$ and an isomorphism $\pi:G_1/N_1\to G_2/N_2$ define a
bitransitive $G_1$-$G_2$-biset. 

Two triples $(N_1,N_2,\pi)$ and
$(N_1',N_2',\pi')$ define equivalent bitransitive $G_1$-$G_2$-bisets  if and only
if $N_1=N_1', N_2=N_2'$ and there exists $b \in G_2/N_2$ such that
$\pi(a)=b\pi'(a)b^{-1}$, for all $a\in G_1/N_1$.
\end{proposition}
\begin{proof}
From the previous discussion, a right transitive $G_1$-$G_2$-biset can be identified with a pair
$(H,\pi)$, where $H=\Stab_r(x)$ and $\pi:G_1\to \text{N}_{G_2}(H)/H$ is
a group morphism. Since $\Stab_l(x)=\ker(\pi)$, the subgroup $\Stab_l(x)$ is normal. In an analogous way, $\Stab_r(x)$ is normal by the bitransitivity of $X$. The map $\pi$ induces an injective homomorphism  $\tilde{\pi} :G_1/\Stab_l(x)\to G_2/\Stab_r(x)$. Since $X$ is bitransitive,  $\operatorname{Im}(\pi)$ is a transitive group with respect to $X$. Therefore $\pi$ and  $\tilde{\pi}$ are surjective.

Next, given a pair of normal subgroups $N_1 \unlhd G_1, N_2\unlhd
G_2$ and an isomorphism $\pi:G_1/N_1\to G_2/N_2$, the bitransitive biset associated is $X=G_2/N_2$ as a right $G_2$-set and a left $G_1$-action is given by $g_1\cdot g_2N_2=\pi(g_1)g_2N_2$, for all $g_1\in G_1, g_2\in G_2$. 

If  $(N_1,N_2,\pi)$ and
$(N_1',N_2',\pi')$ define isomorphic bisets, then 
$N_1=N_1', N_2=N_2'$ since they are normal subgroups obtained as the stabilizers of any  $x\in X$. Moreover, since $G_1/N_1\cong G_2/N_2$  as right $G_1$-sets, there exists $b \in G_2/N_2$ such that $\pi(a)=b\pi'(a)b^{-1}$, for all $a\in G_1/N_1$.
\end{proof}

\subsection{Preliminaries on group cohomology}\label{Preliminares de cohomologia}
Let $G$ be a finite group, $X$ a left $G$-set, and $\omega \in
Z^3(G,k^{\times})$ a 3-cocycle on $G$.  Denote by $C^{n}(G,C^m(X,k^{\times}))$
the abelian group of all maps
$$\beta: \underbrace{ G\times \ldots
\times G}_{n \text{-times}} \times\underbrace{ X\times \ldots \times
X }_{m \text{- times}}\to k^{\times}$$ such that
$\beta(\sigma_1\ldots,\sigma_n;x_1,\ldots,x_m)=1$ if some $\sigma_i$ or $x_i$ is
$1$.

Next define $\delta_G:C^{n}(G,C^m(X,k^{\times}))\to
C^{n+1}(G,C^m(X,k^{\times}))$ in the following way:
\begin{align*}
\delta_G(f)(\sigma_1,\ldots,\sigma_n,\sigma_{n+1};x_1, &\ldots,x_m) =f(\sigma_2,\ldots,\sigma_{n+1};x_1,\ldots,x_m)\\
&\times\prod_{i=1}^n
f(\sigma_1,\ldots,\sigma_i\sigma_{i+1},\ldots,\sigma_{n+1};x_1,\ldots,x_m)^{(-1)^{i}}
\\ & \times
f(\sigma_1\ldots,\sigma_n;\sigma_{n+1}x_1,\ldots,\sigma_{n+1}x_m)^{(-1)^{n+1}}.
\end{align*}
In general $C^n(G,k^{\times}):=C^{n}(G,C^0(X,k^{\times}))\subset
C^{n}(G,C^m(X,k^{\times}))$ as constant functions over $X\times \cdots
\times X$.

Given  $f\in C^{n+1}(G,C^1(X,k^{\times}))$, we define $$Z^n_G(X,k^{\times})_f:=\{\alpha\in C^{n}(G,C^1(X,k^{\times}))| \delta_G(\alpha)=f
\}$$ and $B^n_G(X,k^{\times})=\{\delta_G(\beta)|\beta \in
C^{n-1}(G,C^1(X,k^{\times}))  \}$. The elements of $Z^n_G(X,k^{\times})_f$ are called \emph{$f$-twisted $n$-cocycles} and the elements in
$B^n_G(X,k^{\times})$ are called \emph{n-coboundaries}.

The abelian group $B^n_G(X,k^{\times})$ acts on $Z^n_G(X,k^{\times})_f$
by multiplication. The set of orbits  $Z^n_G(X,k^{\times})_f/B^n_G(X,k^{\times})$ is denoted by $H^n_G(X,k^{\times})_f$ and two
$f$-twisted  $n$-cocycles in the same orbit are called \emph{cohomologous}.

The following is an $f$-twisted version of Shapiro's Lemma.
\begin{proposition}\label{shapiro's lemma}
Let $H\subset G$ be a subgroup and consider $X:=G/H$. Given $f\in C^{n+1,1}(X\rtimes G,k^{\times})$, the set
$H^n_G(X,k^{\times})_f$ admits a natural free and transitive  action
by the abelian group $H^n(H,k^{\times})$. Hence, either $H^n_G(X,k^{\times})_f=\emptyset $
or there is a bijection between $H^n_G(X,k^{\times})_f$ and $H^n(H,k^{\times})$. This bijection depends on the choice of
a particular element of $H^n_G(X,k^{\times})_f$.
\end{proposition}
\begin{proof}
Since $H^n_G(X,k^{\times})=H^n(G,\text{Ind}_H^G(k^{\times}))$, it follows from Shapiro's Lemma that $H^n_G(X,k^{\times}) \cong
H^n(H,k^{\times})$.
It is easy to see that $H^n_G(X,k^{\times})$ acts freely and
transitively over $H^n_G(X,k^{\times})_f$ by multiplication. Therefore, $H^n(H,k^{\times})$ acts freely and transitively over
$H^n_G(X,k^{\times})_f$.
\end{proof}
\begin{remark}
If $X$ is a transitive $G$-set, then  $H^n_G(X,k^{\times})_f=\emptyset $ if and only if
$0\neq [f|_{\Stab_x(G)^{\times n}}]\in H^n(\Stab_x(G),k^{\times})$,
for some $x\in X$.
\end{remark}

\subsection{Bimodule categories over pointed fusion categories}

In this subsection, we will follow the notation in Appendix \ref{Apendice}. 

A $\vect_{G_1}^{\omega_1}$-$\vect_{G_2}^{\omega_2}$-bimodule category is, by
definition, a left module category over $$\vect_{G_1}^{\omega_1}\boxtimes
(\vect_{G_2}^{\omega_2})^{\op}= \vect_{G_1\times G_2}^{\omega_1\times
\omega_2^{-1}}.$$ A  explicit definition is the
following:
\begin{definition}
A $\vect_{G_1}^{\omega_1}$-$\vect_{G_2}^{\omega_2}$-bimodule category is
determined by the data $(X,\mu_l,\mu_r,\mu_m)$, where $X$ is a
$G_1$-$G_2$-biset, $(X,\mu_l)$ is a left $\vect_{G_1}^{\omega_1}$-module
category, $(X,\mu_r)$ is a right $\vect_{G_2}^{\omega_2}$-module category
and $\mu_m: G_1\times X\times G_2\to k^{\times}$ is a normalized map such
that
\begin{align}\label{3-cociclo bimodule1}
\mu_r(\sigma x,\rho,\phi)\mu_m(\sigma,x,\rho\phi)&=\mu_m(\sigma,x,\rho)\mu_m(\sigma,x \rho,\phi)\mu_r(x,\rho,\phi),\\
\mu_m(\sigma\tau,x,\phi)\mu_l(\sigma,\tau,x\phi)&=\mu_l(\sigma,\tau,x)\mu_m(\sigma,\tau
x,\phi)\mu_m(\tau,x,\phi),\label{3-cociclo bimodule2}
\end{align}
for all $\sigma, \tau \in G_1, \rho, \phi\in G_2, x\in X$.

We will denote by $\M(X,\mu_l,\mu_r,\mu_m)$ the $\vect_{G_1}^{\omega_1}$-$\vect_{G_2}^{\omega_2}$-bimodule category associated to $(X,\mu_l,\mu_r,\mu_m)$.
\end{definition}

By \cite[Proposition 4.2]{ENO3}, if $\M(X,\mu_l,\mu_r,\mu_m)$ is an
invertible $\vect_{G_1}^{\omega_1}$- $\vect_{G_2}^{\omega_2}$-bimodule category
then $X$ is a bitransitive $G_1$-$G_2$-biset.

Let $X$ be a $G_1$-$G_2$-biset and $\omega_i\in Z^3(G_i,k^{\times})$. By Proposition \ref{shapiro's lemma}, there is a bijective
correspondence between elements in $Z^2_{G_1\times
G_2}(X,k^{\times})_{\omega_1\times \omega_2^{-1}}$ and all
possible triples $(\mu_l,\mu_r,\mu_m)$ such that
$(X,\mu_l,\mu_r,\mu_m)$ is a $\vect_{G_1}^{\omega_1}$-
$\vect_{G_2}^{\omega_2}$-bimodule category.

If $X$ is bitransitive with associated data $(N_1,N_2,f)$, the
bi-stabilizer of $X$ is $G_1\times_{f}G_2=\{ (g_1,g_2)|
f(g_1N_1)=g_2N_2\}$. By Proposition \ref{shapiro's lemma},
when $H^2_{G_1\times G_2}(X,k^{\times})_{\omega_1\times \omega_2^{-1}}\neq
\emptyset$ there is a bijective correspondence between the set of
equivalence classes of bimodule categories with underlying $G_1$-$G_2$-biset $X$
and $H^2(G_1\times_{f}G_2,k^{\times})$.

\subsection{Parametrization of invertible bimodule categories over pointed fusion categories and braided equivalence of twisted Drinfeld doubles.}

\begin{lemma}\label{lema pairing}
Let $G_1$ and $G_2$ be finite groups of the same order and $\omega_i\in Z^3(G_i,k^{\times})$.
Let $\M(X,\mu_l,\mu_r,\mu_m)$ be a
$\vect_{G_1}^{\omega_1}$-$\vect_{G_2}^{\omega_2}$-bimodule category, with $X$ a
bitransitive set.

For a fix $x\in X$, every $\mu = [(\mu_l,\mu_r,\mu_m)] \in H^2_{G_1\times
G_2}(X,k^{\times})_{\omega_1\times \omega_2^{-1}}$ defines a group morphism
\begin{align*}
L_\mu :\Stab_{l}(x) &\to H^1_{G_2}(X,k^{\times})\\
n_1 &\mapsto  \Big[(y,g_2)\mapsto \mu_m(n_1,y,g_2)\Big],
\end{align*}
and induces a pairing \begin{align*}
\mu_m(-,x,-):\Stab_{l}(x) \times \Stab_{r}(x) &\to k^{\times} \\
(n_1,n_2) &\mapsto \mu_m(n_1,x,n_2).
\end{align*}
Moreover, $L_\mu$ is a group isomorphism if and only if $\mu_m(-,x,-)$
is non-degenerate. In particular, $\Stab_{l}(x)$ and $\Stab_{r}(x)$
are normal abelian subgroups when $L_{\mu}$ is an isomorphism.
\end{lemma}
\begin{proof}
It follows from equation \eqref{3-cociclo bimodule1} that
$(x,n_2)\mapsto \mu_m(n_1,x,n_2)$ is a 1-cocycle in $Z_{G_2}^1(X,k^{\times})$. By equation
\eqref{3-cociclo bimodule2}, $\mu_m(n_1n_1',-,-)$ is cohomologous to
$\mu_m(n_1,-,-)\times \mu_m(n_1',-,-)$, for all $n_1,n_1'\in \Stab_l(x)$. 

By
Shapiro's Lemma, $L_\mu$ is completely determined by the group
morphism $\Stab_l(x)\to \Hom(\Stab_r(x),k^{\times}), n_1\mapsto  \mu_m(n_1,x,-)$.

From Remark \ref{observacion ordenes de estabilizadores} (1), we have that $|\Stab_{l}(x)| = |\Stab_{r}(x)|$. In
this way, $\mu_m(-,x,-)$ defines a pairing and $L_\mu$ is an
isomorphism if and only if $\mu_m(-,x,-)$ is non-degenerate.
\end{proof}
The next theorem is a generalization of \cite[Corollary 3.6.3]{Davy2} and  \cite[Proposition 5.2]{NR}.
\begin{theorem}\label{Parametrizacion invertible}
Let $G_1$ and $G_2$ be finite groups and  $\omega_i\in
Z^3(G_i,k^{\times})$. Let $\M(X,\mu_l,\mu_r,\mu_m)$ be a
$\vect_{G_1}^{\omega_1}$-$\vect_{G_2}^{\omega_2}$-bimodule category. Then, the
bimodule $\M(X,\mu_l,\mu_r,\mu_m)$ is invertible if and only if
\begin{enumerate}
\item $X$ is bitransitive, and
\item $\mu_m(-,x,-)$ is non-degenerate.
\end{enumerate}
\end{theorem}
\begin{proof}
If $\M(X,\mu_l,\mu_r,\mu_m)$ is invertible then $\vect_{G_1}^{\omega_1}$
and $\vect_{G_2}^{\omega_2}$ are Morita equivalent. Thus, by
\cite[Theorem 2.15]{ENO}, $|G_1| = |G_2|$. It follows from \cite[Proposition 4.2]{ENO3} that if
$\M(X,\mu_l,\mu_r,\mu_m)$ is an invertible $\vect_{G_1}^{\omega_1}$-
$\vect_{G_2}^{\omega_2}$-bimodule category then $X$ is bitransitive. Again by \cite[Proposition 4.2]{ENO3}, the bimodule
$\M(X,\mu_l,\mu_r,\mu_m)$ is invertible if and only if the group
morphism induced by left multiplication of objects of
$\vect_{G_1}^{\omega_1}$ $$L:G_1\to
\Aut_{\vect_{G_2}^{\omega_2}}(\M(X,\mu_l))$$ is an isomorphism.

Now assume that $X$ is bitransitive with data $(N_1,N_2,f)$. Then without loss of generality we can
suppose that $X=G_2/N_2$ with action $g_1 \overline{a} g_2= f(g_1)
\overline{a g_2}$, for $g_1\in G_1, g_2\in G_2, \overline{a}\in
X$. Considering the exact sequence \eqref{sucecion Aut} in this case, we have that:
\begin{align*}
& \xymatrix{  1 \ar[r]^{}  & N_1 \ar[r]^{} \ar[dd]^{L_\mu} & G_1
\ar[r]^{}\ar[dd]^{L} & G_1/N_1\ar[r]^{}\ar[dd]^{f} & 1 \\  &
& & & \\ 1 \ar[r]^{}  & H^1_{G_2}(X,k^{\times}) \ar[r]_{} &
\Aut_{\vect_{G_2}^{\omega_2}}(\M(X,\mu_l)) \ar[r]_{} &
G_2/N_2\ar[r]^{} & 1.}
\end{align*}
Hence, if $X$ is bitransitive, $\M(X,\mu_l,\mu_r,\mu_m)$ is
invertible if and only if the group morphism $L_\mu:N_1\to
H^1_{G_2}(X,k^{\times})$ is an isomorphism. Thus, by Lemma \ref{lema
pairing}, $\M(X,\mu_l,\mu_r\mu_m)$ is invertible if and only if
$\mu(-,x-)$ is non-degenerate.
\end{proof}
\begin{remark}
Let $\M(X,\mu_l,\mu_r,\mu_m)$ be a
$\vect_{G_1}^{\omega_1}$-$\vect_{G_2}^{\omega_2}$-bimodule category with $X$
bitransitive $G_1$-$G_2$-biset. Let $(N_1,N_2,f)$ be the data associated to the
bitransitive $G_1$-$G_2$-biset $X$. Then, the cohomo\-lo\-gy class of
$\omega_1\times\omega_2^{-1}|_{G_1\times_f G_2}$ is trivial
and we can assume that
$\omega_1\times\omega_2^{-1}|_{G_1\times_f G_2}=1$. There is a canonical correspondence between the set
$H^2_{G_1\times G_2}(X,k^{\times})_{\omega_1\times \omega_2^{-1}}$ and the
set $H^2(G_1\times_f G_2)$. Moreover, an element $\psi\in
Z^2(G_1\times_f G_2, k^{\times})$ defines an invertible
bimodule category if and only if the pairing $\psi(-|-):N_1\times
N_2\to k^{\times}$ is non-degenerate.
\end{remark}

\begin{corollary}
Let $G_1$ and $G_2$ be finite groups and $\omega_i\in Z^3(G_i,k^{\times})$.
Let $\M(X,\mu_l,\mu_r,\mu_m)$ be an invertible
$\vect_{G_1}^{\omega_1}$-$\vect_{G_2}^{\omega_2}$-bimodule category with
$(N_1,N_2,f)$ the data associated to the bitransitve biset $X$. Then
\begin{enumerate}
\item The subgroups $N_i\unlhd G_i$ are normal and abelian.
\item There is a group isomorphism $N_1\cong N_2$.
\item $\M(X,\mu_l)$ and $\M(X,\mu_r)$ are pointed module categories (see Definition \ref{definicion pointed module category}).
\end{enumerate}
\end{corollary}
\begin{remark}
Let $G_1$ and $G_2$ be finite groups and
$\omega_i\in Z^3(G_i,k^{\times})$. There is a
bijective correspondence between \textit{braided} tensor equivalences from
$\mathcal{Z}(\vect_{G_1}^{\omega_1})$ to $\mathcal{Z}(\vect_{G_2}^{\omega_2})$
and invertible $\vect_{G_1}^{\omega_1}$-$\vect_{G_2}^{\omega_2}$-bimodule
categories,  \cite[Theorem 1.1]{ENO3}. Then, Theorem \ref{Parametrizacion invertible} gives a
parametrization of the braided tensor equivalences of the category
of representations of twisted Drinfeld modules. This  generalizes  \cite{Davy2}.
\end{remark}

\begin{corollary}
There is a correspondence between elements in $\operatorname{BrPic}(\vect_{G}^{\omega})$ and equivalence classes of quadruples $(A_1,A_2,f,\mu)$, where $A_1$ and $A_2$ are normal abelian isomorphic subgroups of $G$, $f:G/A_1\to G/A_2$ is a group
isomorphism and $\mu \in Z^2_{G\times G}(X,k^{\times})_{\omega\times
\omega^{-1}}$ such that the pairing $\mu_m(-,x,-):A_1\times A_2\to
k^{\times}$ is non-degenerate, where $X$ is the bitransitive $G$-biset associated to $(A_1,A_2,f)$. 

Two quadruples $(A_1,A_2,f,\mu)$ and $(A_1',A_2',f',\mu')$ are equivalent if $A_1=A_1'$, $A_2=A_2'$ and there is $b\in G/A_2$ such that $f(a)=bf'(a)b^{-1}$, for all $a\in G/A_1$ and $[\mu]= [\mu'^b]\in H^2_{G\times G}(X,k^{\times})_{\omega\times
\omega^{-1}}$.
\end{corollary}
\subsection{Tensor product of module categories over pointed fusion categories}

\subsubsection{Equivariantization of semisimple  categories}\label{seccion equivariantizacion semisimple}

Let $\M(X,\alpha)$ be a left $\vect_G$-module category, see Appendix \ref{Apendice}. Let $k^X$ be
the algebra of functions from $X$ to $k$. We will denote by
$\{e_x\}_{x\in X}$ the basis of $k^X$ formed by the orthogonal
primitive idempotents. We further define the $G$-crossed product algebra $ G\#_\alpha k^X$, with
basis given by $\{ g\#e_x \}_{x\in X, g\in G}$ and multiplication
$$(g\#e_s)  (h\#e_t )=  gh\# \delta_{s,ht}\alpha(g,h;t)e_t.$$
The category of right $G\#_\alpha
k^X$-modules is exactly the category $\M(X,\alpha)^{G}$ of
$G$-equivariant objects. In fact, if $V$ is a  $G\#_{\alpha}
k^X$-module then $V=\oplus_{x\in X}V_x$, where $V_x=\{v(1\# e_x)|v\in
V \}$, and the linear isomorphisms
\begin{align*}
f(\sigma,x): V_{\sigma x} &\to V_{x}\\
v &\mapsto v(\sigma\# e_x),
\end{align*}
satisfy $$f(\tau,x)\circ f(\sigma,\tau
x)=\alpha(\sigma,\tau;x)f(\sigma\tau,x),$$ for all $\sigma,\tau\in
G, x\in X$. Thus, $(V,f_\sigma:=\oplus_{x\in
X}f(\sigma,x):\sigma\otimes V\to V)_{\sigma\in G}$ is an object in
$\M(X,\alpha)^{G}$.

Given $x\in X$, we will denote by $\mathcal{O}(x)$ the orbit of $x$ in $X$. If $\{x_1,\ldots,x_m\}$ is a set of representatives of the orbits, we have
$$ G\#_\alpha k^X=\bigoplus_{i=1}^n G\#_\alpha k^{\mathcal{O}(x_i)},$$
and $G\#_\alpha k^{\mathcal{O}(x_i)}$ are mutually orthogonal
bilateral ideals.

In order to describe the simple objects in $\M(X,\alpha)^{G} = \Rep (G\#_{\alpha} k^X)$, we 
assume that $X=\mathcal{O}(x)$, for some $x\in X$.  The map
$\alpha_x:=\alpha(-,-,x):\Stab(x)\times \Stab(x)\to k^{\times}$ is a
2-cocycle. By the Clifford
theory for crossed products (see \cite{Karpy}), there is an equivalence between
the category of right modules over the twisted group algebra
$k_{\alpha_x}[ \Stab(x)]$ and the category of right modules over
$G\#_\alpha k^X$. If $U$ is a right $k_{\alpha_x}[ \Stab(x)]$-module,
the action $ u (e_y\# \sigma)= \delta_{y,x} u\cdot \sigma$ defines a right
$\Stab(x)\#_\alpha k^X$-module structure on $U$ and $\Ind_{\Stab(x)\#_\alpha
k^X}^{G\#_\alpha k^X}(U)$ is the associated  $G\#_\alpha
k^X$-module. Conversely, if $V$ is a $G\#_\alpha k^X$-module then
$V_x$ is a $k_{\alpha_x}[ \Stab(x)]$-module with action given by
$(v_x)g =(v_x)g\#e_x$, for all $v_x\in V_x, g\in \Stab(x)$.

\subsubsection{Tensor product of module categories over pointed fusion categories as an equivariantization}

Let $\C$ be a fusion category and $\M$ a  $\C$-bimodule category.
The following definition was given in \cite{GNN}. The {\em center}
of $\M$ is the category $\Z_\C(\M)$ where objects are pairs $(M,\,
\gamma)$, with $M$ an object of $\M$ and
\begin{equation}
\label{gamma} \gamma = \{ \gamma_{X} : X\ot M \xrightarrow{\sim}  M
\ot X \}_{ X \in \C}
\end{equation}
a natural family of isomorphisms making the following diagram
commutative:
\begin{equation}
\label{central object} \xymatrix{ & X \ot (M \ot Y)
\ar[rr]^{\alpha_{X,M,Y}^{-1}} & &
(X \ot M) \ot Y  \ar[dr]^{\gamma_{X}} & \\
X\ot (Y\ot M) \ar[ur]^{\gamma_{Y}}
\ar[dr]_{\alpha_{X,Y,M}^{-1}} & &  &  &(M \ot X) \ot Y, \\
& (X\ot Y)\ot M \ar[rr]_{\gamma_{X\ot Y}} & & M\ot (X\ot Y)
\ar[ur]_{\alpha_{M,X,Y}^{-1}} & }
\end{equation}
where $\alpha$ denotes the corresponding associativity constraint in
$\M$.

The family of natural isomorphisms \eqref{gamma} is called a {\em
central structure} of an object $ M \in \Z_\C(\M)$.

If $\M(X,\mu_l,\mu_m,\mu_r)$ is a $\vect_{G}^{\omega}$-bimodule category
then a central structure on $V=\oplus_{x\in X}V_x$ is given by a
family of linear isomorphism $\gamma(g,x):V_{gx}\to V_{xg}$ such
that
$$\mu_{m}(\sigma,x,\tau)\gamma(\sigma\tau,x)=\mu_l(\sigma,\tau;x)\mu_r(x;\sigma,\tau)\gamma(\sigma,x)\circ
\gamma(\tau,x),$$ for all $\sigma,\tau \in G, x\in X$.

Since $\M(X,\mu_l,\mu_m,\mu_r)$ is a left $\vect_{G\times G}^{\omega\times \omega^{-1}}$-module category and the diagonal inclusion
$\Delta:\vect_G\to \vect_{G\times G}^{\omega\times \omega^{-1}}$ is a
strict tensor functor, then $\M(X,\mu_l,\mu_m,\mu_r)$ is a left
$\vect_G$-module category with action $g\cdot x:= gxg^{-1}$ and
2-cocycle
$$\alpha(a,b;x):=\frac{\mu_m(a,bx,b^{-1})\mu_l(a,b;x)}{\mu_r(abx;b^{-1},a^{-1})},$$
for all $a,b\in G, x\in X$.
\begin{proposition}\label{center of bimodule over pointed}
The equivariantization $\M(X,\mu_l,\mu_m,\mu_r)^G$  is canonically
 equivalent to the center $\Z_{\vect_{G}^{\omega}}(\M(X,\mu_l,\mu_m,\mu_r))$.
\end{proposition}
\begin{proof}
If $V=\oplus_{x\in G}V_x$ with $\{f(g,x):V_{gxg^{-1}}\to
V_x\}_{x\in X, g \in G}$ is an object in
$\M(\mu_l,\mu_m,\mu_r)^G$ then a central structure on $V$ is
defined by the composition \begin{align*} g\otimes V_{x}=(g\otimes
V_x)\otimes(g^{-1}\otimes g) \xrightarrow{\omega(gx,g^{-1},g)^{-1}}
V_{gxg^{-1}}\otimes g\xrightarrow{f(g,x)\otimes g} V_x\otimes
g=V_{xg}.
\end{align*}
Conversely, if $\gamma(g,x):V_{gx}\to V_{xg}$ is a central structure
on $V$ then the map $f(g,x)=\omega(gx,g^{-1},g)\gamma(g,x)\otimes
g^{-1}$ defines a $G$-equivariant structure on $V$.
\end{proof}
The next result describes the tensor product of bimodule categories in terms of  equivarizations of categories. 
\begin{theorem}\label{tensor product bimodule pointed}
Let $\M(X,\mu^X)$ be a right $\vect_{G}^{\omega}$-module category and
$\M(Y,\mu^Y)$ be a left $\vect_{G}^{\omega}$-module category, then
$\M(X\times Y,\mu^Y,1,\mu^X)$ is a $\vect_{G}^{\omega}$-bimodule category
and $$\M(X\times
Y,\mu^Y,1,\mu^X)^G=\M(X,\mu^X)\boxtimes_{\vect_{G}^{\omega}}
\M(Y,\mu^Y),$$ with a $\vect_{G}^{\omega}$-balanced bifunctor
$$\Ind_{k^{X\times Y}}^{G\# k^{X\times Y}}(-):\M(X,\mu^X)\boxtimes
\M(Y,\mu^Y)\to \M(X\times Y,\mu^Y,1,\mu^X)^G$$
\end{theorem}
\begin{proof}
This follows from Proposition \ref{center of bimodule over pointed} and
\cite[Proposition 3.8]{ENO3}.
\end{proof}

The following corollary gives an alternative description of the
simple objects of $$\Fun_{\vect_{G}^{\omega}}(\M(X,\mu^X),\M(Y,\mu^Y))$$ given by Ostrik in \cite[Proposition 3.1]{ostrik}.
\begin{corollary}\label{corolario alternativo teo de ostrik}
Let $\M(X,\mu_X)$ and  $\M(Y,\mu_Y)$ be left $\vect_{G}^{\omega}$-module
categories. Let $\{(x_i, y_i)\}_{i=1}^n$ be a set of representatives
of the orbits of $G$ in $X\times Y$.

There is a bijective correspondence between simple objects of
$$\Fun_{\vect_{G}^{\omega}}(\M(X,\mu_X),\M(Y,\mu_Y))$$ and irreducible
representations of $k_{\alpha_{(x_i, y_i)}}[\Stab_G((x_i, y_i))]$, where
$$\alpha_{(x_i,y_i)}(g,h)=\mu_X(g,h;x_i)\mu_Y(g,h;y_i),$$ for all
$g,h\in \Stab_G((x_i, y_i))$.
\end{corollary}
\begin{proof}
By \cite[Proposition 3.2]{ENO3} we have that
\begin{align*}
\Fun_{\vect_{G}^{\omega}}(\M(X,\mu_X),\M(Y,\mu_Y)) &\cong \M(X,\mu_X)^{\op}\boxtimes_{\vect_{G}^{\omega}}\M(Y,\mu_Y)\\
 &\cong \big(\M(X^{\op},\mu_X^{\op})\boxtimes \M(Y,\mu_Y) \big)^G,
\end{align*} where $X^{\op}$ denotes the right $G$-set $X$ with action $xg:=g^{-1}x$
and $\mu_{X}^{\op}(x,g,h):=\mu_X(h^{-1},g^{-1};x)^{-1}$, for all
$g,h\in G, x\in X$.

By Proposition \ref{tensor product bimodule pointed}, the simple
objects over $\Fun_{\vect_{G}^{\omega}}(\M(X,\mu_X),\M(Y,\mu_Y))$  are in
correspondence with simple modules over $G\#_\alpha k^{X\times Y}$,
where $G$  acts on $X\times Y$ by $g(x,y)=(gx,gy)$ and the
2-cocycle is given by
$$\alpha(g,h;(x,y))=\mu_X(g,h,ghx)\mu_Y(g,h,y),$$ for all $g,h\in G,
x,\in X,y\in Y$.

It follows from the discussion in Subsection \ref{seccion
equivariantizacion semisimple} that if $\{(x_i, y_i)\}_{i=1}^n$ is a
set of representatives of the orbits of $G$ in $X\times Y$ there is
a bijective correspondence between simple objects of
$\Fun_{\vect_{G}^{\omega}}(\M(X,\mu_X),\M(Y,\mu_Y))$ and simple modules over $k_{\alpha_{(x_i,y_i)}}[\Stab((x_i,y_i))]$, for all $i\in \{1,\ldots,
n\}$.
\end{proof}
\begin{remark}
\begin{enumerate}[leftmargin=*]
\item Suppose that $X$ and $Y$ are transitive $G$-sets
with $X=G/H_1$ and $Y=G/H_2$. Then $G$-orbits of $X\times
Y$ are in correspondence with $(H_1,H_2)$-double cosets. If
$\{g_i\}_{i=1}^n$  is a set of representatives of the double cosets,
the associated stabilizer is $H_1\cap g_iH_2g_i^{-1}$. 
Ostrik's classification of  simple objects of
$\Fun_{\vect_{G}^{\omega}}(\M(X,\mu_X),\M(Y,\mu_Y))$ can be recovered in
this way, \cite[Proposition 3.2]{ostrik}.
\item Theorem \ref{tensor product bimodule pointed}
also gives a description of the tensor product of bimodule
categories over pointed fusion categories. 

If
$\M(X,\mu_l^X,\mu_m^X,\mu_r^X)$ is a
$\vect_{G_1}^{\omega_1}$-$\vect_{G_2}^{\omega_2}$-bimodule category and
$\M(Y,\mu_l^Y,\mu_m^Y,\mu_r^Y)$ is a
$\vect_{G_2}^{\omega_2}$-$\vect_{G_3}^{\omega_3}$-bimodule category then 
$\M(X,\mu_l^X,\mu_m^X,\mu_r^X)\boxtimes_{\vect_{G_2}^{\omega_2}}
\M(Y,\mu_l^Y,\mu_m^Y,\mu_r^Y) \cong \Rep(G_2\# k^{X\times Y})$ . The set $\text{Spec}(G_2\# k^{X\times
Y})$ of isomorphism classes of simple modules is the
$G_1$-$G_3$-biset associated to the tensor product.

The cohomological data can be calculated fixing a set of
representatives of isomorphism classes of the simple modules of $G_2\#
k^{X\times Y}$.
\end{enumerate}
\end{remark}
Ostrik's description of fiber functors over $\C(G,\omega;X,\alpha_X)$ \cite[Corollary 3.1]{ostrik}, can be reformulated using Corollary \ref{corolario alternativo teo de ostrik} as follows. The fiber functors of $\C(G,\omega;X,\alpha_X)$  are classified by equivalence classes of $\vect_{G}^\omega$-module categories $\M(Y,\alpha_Y)$ such that $G$ acts transitively on $X\times Y$ and the twisted group algebra $k_{\alpha_{(x,y)}}[\operatorname{Stab_G((x,y))}]$ is simple for some pair $(x,y)$ (and thus for every pair), where  $\alpha_{(x,y)}(g,h)=\mu_X(g,h,x)\mu_Y(g,h,y)$.  By Tannaka forma\-lism there is a unique (up to isomorphism) Hopf algebra $H(G,\omega;X,\mu_X,Y,\mu_Y)$ such that $\operatorname{Corep(H(G,\omega;X,\mu_X,Y,\mu_Y))}=\C(G,\omega;X,\mu_X)$ and which satisfies a certain universal property, \cite{Tann}. 

The Hopf algebras $H(G,\omega;X,\mu_X,Y,\mu_Y)$ are called group-theoretical and they are very important in the theory of semisimple Hopf algebra, since they include abelian extensions and twisting of groups algebras among others. 

\begin{theorem}
Two group-theoretical Hopf algebras $$H(G,\omega;X,\mu_X,Y,\mu_Y)  \text{ and   } H(G',\omega';X',\mu_{X'},Y',\mu_{Y'})$$ are isomorphic if and only if there exists an invertible $(\vect_{G}^{\omega},\vect_{G'}^{\omega'})$-bimodule category $\mathcal{S}$  such that
\[\mathcal{S}\boxtimes_{\vect_{G'}^{\omega'}}\M(X',\mu_{X'})\cong \M(X,\mu_{X}) \  \text{ and }\  \mathcal{S}\boxtimes_{\vect_{G'}^{\omega'}}\M(Y',\mu_{Y'})\cong \M(Y,\mu_{Y})\] as $\vect_{G}$-module categories.
\end{theorem}
\begin{proof}
Let $H(G,\omega;X,\mu_X,Y,\mu_Y)$ and $H(G',\omega';X',\mu_{X'},Y',\mu_{Y'})$ be group-theoretical Hopf algebras. By Tannaka formalism the Hopf algebras are isomorphic if and only if there is a tensor equivalence  $F:\C(G,\omega;X,\mu_X)\to \C(G',\omega';X',\mu_{X'})$ such that the diagram

\begin{equation}\label{diagrama final}
\xymatrix{\C(G,\omega;X,\mu_X) 
\ar[rr]^{F} \ar[dr]_{U(Y, \mu_{Y})}& &
\C(G',\omega';X',\mu_{X'}) \ar[dl]^{U(Y', \mu_{Y'})} & \\
& \vect & }
\end{equation}
commutes, where $U(Y,\mu_{Y})$ ($U(Y',\mu_{Y'})$, respectively) is the fiber functor of $\C(G,\omega;X,\mu_X)$ ($\C(G',\omega';X',\mu_{X'})$, respectively) associated to the left (right, respectively) rank one module category $\Fun_{\vect_{G}^\omega}(\M(X,\mu_X), \M(Y,\mu_Y))$ ($\Fun_{\vect_{G'}^{\omega'}}(\M(X',\mu_{X'}), \M(Y',\mu_{Y'}))$, respectively). 

By Proposition \ref{equivalencias en Cor}, tensor equivalences from $\C(G,\omega;X,\mu_X)$ to $\C(G',\omega';X',\mu_{X'})$ are in correspondence with $(\vect_{G}^{\omega},\vect_{G'}^{\omega'})$-bimodule categories $\mathcal{S}$  such that
\[\mathcal{S}\boxtimes_{\vect_{G'}^{\omega'}}\M(X',\mu_{X'})\cong \M(X,\mu_{X})\] as $\vect_{G}$-module categories. 

By Theorem \ref{description of Funct}, the diagram \eqref{diagrama final} commutes if and only if
\[\mathcal{S}\boxtimes_{\vect_{G'}^{\omega'}}\M(Y',\mu_{Y'})\cong \M(Y,\mu_{Y})\] as $\vect_{G}$-module categories.

\end{proof}
\section{Appendix: Module categories over pointed fusion categories}\label{Apendice}

Recall that a fusion category is called
\emph{group-theoretical} if it is
Morita equivalent to a pointed fusion category.
See \cite{ostrik, ENO} for more details about group-theoretical fusion categories.

In this appendix, we use the theory of $G$-sets to give alternative descriptions to the characterization
of indecomposable mo\-dule categories over pointed fusion categories \cite{ostrik1},
group-theoretical fusion categories \cite{ENO} and pointed module categories \cite{Naidu},
using the theory of $G$-sets. These alternative descriptions  are useful, for example, to describe  tensor products of bimodule categories (see Section \ref{Seccion bimodules de pointed fusion cat}).
We will follow the notation of Subsection \ref{Preliminares de cohomologia}.

\subsection{The 2-category of $(G,\omega)$-sets with twist.}
In this subsection, we will define  a 2-category biequivalent to the  2-subcategory of
$\mathfrak{M}_l(\vect_G^\omega)$, where objects are all left module categories over
$\vect_G^\omega$, $1$-cells are  $\vect_G^\omega$-module functors that maps simple objects to simple objects and
$2$-cells are $\vect_G^\omega$-linear natural isomorphisms.

We will fix a finite group $G$ and a $3$-cocycle $\omega\in
Z^3(G,\Tt)$, that is a function $\omega:G\times G\times G\to \Tt$
such that
$$\omega(\sigma\tau,\rho,\phi)\omega(\sigma,\tau,\rho\phi)=\omega(\sigma,\tau,\rho)\omega(\sigma,\tau\rho,\phi)\omega(\tau,\rho,\phi),$$
for all $\sigma, \tau, \rho, \phi\in G$.  

The $2$-category
of finite $(G,\omega)$-sets with twist is defined as follows:
\begin{enumerate}[leftmargin=*]
\item Objects are pairs
$(X,\alpha)$, where $X$ is a finite left $G$-set and $\alpha\in
Z^2_G(X,\Tt)_\omega$. They will be called \emph{finite $(G,\omega)$-sets
with twist}.
\item Let $(X,\alpha_X),$ $(Y,\alpha_Y)$ be finite $(G,\omega)$-sets with twist. A $1$-cell from
$(X,\alpha_X)$ to  $(Y,\alpha_Y)$, also called a \emph{$G$-equivariant
map}, is a pair $(L,\beta)$,
$$\xymatrix{(X,\alpha_X) \ar[r]^{(L,\beta)} &(Y,\alpha_Y) },$$ where
\begin{itemize}
\item  $L: X\to Y$ is a morphism of left $G$-sets,
\item $\beta\in C^{1}(G,C^1(X,k^{\times}))$ such
that $\delta_G(\beta)= L^*(\alpha_Y)(\alpha_X)^{-1}$, that is a map
$$\beta:G\times X\to  \Tt$$
\end{itemize}such that $$\beta(\tau; x) \beta(\sigma\tau;x)^{-1} \beta(\sigma; \tau x)=\alpha_Y(\sigma,\tau;L(x))
\alpha_X(\sigma,\tau;,x)^{-1},$$ for all $\sigma, \tau \in G, x\in
X$.
\item Given two 1-cells $(L,\beta), (L',\beta'): (X,\alpha_X)\to (Y,\alpha_Y)$, a $2$-cell is a map $\theta:(L,\beta)\Rightarrow
(L',\beta')$
such that $\delta_G(\theta)= \beta'\beta^{-1}$, that is
$\theta: X\to \Tt$ such that $$\theta(x) \theta(\sigma x)^{-1}
=\beta'(\sigma;x)\beta(\sigma;x)^{-1},$$ for all $\sigma\in G, x\in
X$.
\end{enumerate}

Let $(F,\beta_F): (X,\alpha_X)\to (Y,\alpha_Y)$ and
$(K,\beta_K):(Y,\alpha_Y)\to (Z,\alpha_Z)$ be two 1-cells. Their composition is defined by
$$(K,\beta_K)\circ (F,\beta_F)=(K\circ F,
F^*(\beta_K)\beta_F):(X,\alpha_X)\to (Z,\alpha_Z).$$

If $\theta: (L,\beta) \Rightarrow (L', \beta')$ and $\theta':
(L',\beta') \Rightarrow (L'', \beta'')$ are $2$-cells, their
composition is given by the product of the maps, namely
$$\theta'\circ \theta:=  \theta'\theta: (L,\beta)\Rightarrow (L'',\beta'').$$

Given a twisted $(G,\omega)$-set $(X,\mu),$ we can associate to it a left
$\vect_G^\omega$-module category $\M(X,\mu)$. As a $k$-linear category $\M(X,\mu)$ is the category of
$X$-graded vector spaces. The $\vect_G^\omega$-action is the following
\begin{align*}
\otimes :\vect_G^\omega\boxtimes \M(X,\mu) &\to \M(X,\mu)\\
 k_\sigma\boxtimes k_x &\mapsto k_{\sigma x},
\end{align*} and associativity constraints are $$\mu_{\sigma,\tau, x}=\mu(\sigma,\tau,x)\id_{k_{(\sigma\tau)x}}:(k_\sigma\otimes k_\tau)\otimes k_x\to k_\sigma\otimes( k_\tau\otimes k_x),$$
for all $\sigma,\tau \in G, x\in X$. 

A straightforward calculation
implies the following result.
\begin{proposition}
The $2$-category of twisted $(G,\omega)$-sets is biequivalent to the  2-subcategory of $\mathfrak{M}_l(\vect_G^\omega)$ where objects are module categories, 1-arrows are module functors that maps simple objects to simple objects, and 2-arrows are module natural isomorphisms.
\end{proposition}

In the literature group-theoretical fusion categories are usually
parameterized by data $(G,\omega, H,\psi)$, where $G$ is a finite
group, $\omega$ is a $3$-cocycle in $G$, $H\subset G$ is a subgroup
of $G$ and $\psi\in C^2(H,k^{\times})$ such that
$\delta(\psi)=\omega|_{H^{\times 3}}$. The group
theoretical fusion category $\C(G,\omega,H,\psi)$ is realized as the
category of $k_\psi[H]$-bimodules in $\vect_G^\omega$, see \cite{ostrik}. 

There is also an alternative description of group-theoretical fusion categories in terms of $G$-sets. Explicitly, group-theoretical fusion categories can be parametrized by the data
$(G,\omega,X,\mu)$, where $G$ is a finite group, $\omega \in
Z^3(G,k^{\times})$, $X$ is a transitive left $G$-set, and $\mu \in
Z^2_G(X,k^{\times})_\omega.$  We will denote by $\C(G,\omega,X,\mu):=
\END_{\vect_G^\omega}(\M(X,\mu))$ the group-theoretical fusion
category associated to these data.

Given $(G,\omega,X,\mu)$ and an element $x\in X$, the subgroup
$H:=\Stab_x(G)$ and the 2-cochain $\psi(h,h'):=\mu(h,h',x) $ define a
data $(G,\omega,H,\psi)$ such that $\C(G,\omega,X,\mu)\cong
\C(G,\omega,H,\psi)$ as fusion categories.

The group $\operatorname{BrPic}(\vect_G^\omega)$ acts on the set of equivalence classes of indecomposable $\vect_{G}^\omega$-module categories. The following proposition shows that oribts of this action correspond to equivalence classes of group-theoretical fusion categories over $\vect_{G}^\omega$.
\begin{corollary}
The fusion category $\C(G,\omega,X,\mu)$ is tensor equivalent to $ \C(G',\omega',X',\mu')$ if and only if there is an invertible  $\vect_G^{\omega}$-$\vect_{G'}^{\omega}$-bimodule category $\mathcal{S}$ such that \[\mathcal{S}\boxtimes_{\vect_{G'}^{\omega'}}\M(X',\alpha')\cong \M(X,\mu)\] as $\vect_{G}^{\omega}$-module categories.
\end{corollary}
\begin{proof}
This follows immediately from Proposition \ref{equivalencias en Cor}.
\end{proof}

\subsection{Pointed module categories over $\vect_G^\omega$}\label{Subseccion module cat}
Let $\M(X,\alpha)$ be a left $\vect_G^\omega$-module category.  The group $\Aut_{\vect_G^\omega}(\M(X,\alpha))$ is the group
$\Inv(\C(G,\omega,X,\mu))$ of isomorphism classes of  invertible objects of $\C(G,\omega,X,\mu)$.

Note that every autoequivalence of module categories maps simple objects to simple simple objects, thus  $\Aut_{\vect_G^\omega}(\M(X,\alpha))$ is canonically isomorphic to the group of isomorphism classes of autoequivalence of the twisted $(G,\omega)$-set $(X,\alpha)$. The group $\Aut_G(X)$ of $G$-automorphism of $X$ acts naturally on
$H^2_G(X,\Tt)_\omega$. We will denote by
$\Aut_G(X,[\alpha])\subset \Aut_G(X)$ the stabilizer of $[\alpha]\in
H^2_G(X,\Tt)_\omega$. Thus, the group
$\Aut_{\vect_{G}^{\omega}}(\M(X,\alpha))$  fits into the exact
sequence
\begin{equation}\label{sucecion Aut}
1\to H^1_G(X,\Tt) \to
\Aut_{\vect_{G}^{\omega}}(\M(X,\alpha))\to
\Aut_G(X,[\alpha])\to 1.
\end{equation}
\begin{remark}\label{obs sobre auto y H1 de X transitivo}
When $X$ is a transitive $G$-set, Shapiro's Lemma defines a
group isomorphism $H^1_G(X,\Tt)\cong \Hom(\Stab(x),\Tt)$ and also $\Aut_G(X)\cong N_G(\Stab(x))/\Stab(x)$, for any $x\in X$.
\end{remark}

\begin{definition}\label{definicion pointed module category}\cite{Naidu}
Let $\C$ be a fusion category. A left $\C$-module category $\M$ is
called \emph{pointed} if the dual category $\C_\M^*$ is pointed.
\end{definition}
The following proposition gives an alternative but equivalent
description of  pointed module categories to the one given in \cite{Naidu}.
\begin{proposition}\label{prop clasificacion pointed module}
An indecomposable  $\vect_{G}^{\omega}$-module category $\M(X,\alpha)$ is pointed if and
only if
\begin{enumerate}
\item $[\alpha]$ is invariant under the action of $\Aut_G(X)$, and
\item the stabilizer of $X$ is a normal abelian subgroup of $G$.
\end{enumerate}
\end{proposition}
\begin{proof}
Let $F$ be the stabilizer of a point
$x\in X$. Then $H^1_G(X,\Tt)\cong \Hom(F,\Tt)$ and $\Aut_G(X)\cong
N_G(F)/F$, see Remark \ref{obs sobre auto y H1 de X transitivo}.  From the exactness of the sequence
\eqref{sucecion Aut} it follows that
\begin{align*}
|\Aut_{\vect_{G}^{\omega}}(\M(X,\alpha))| &= |H^1_G(X,\Tt)||\Aut_G(X,[\alpha])|\\
&\leq  |\Hom(F,\Tt)||\Aut_G(X)|\\
& \leq |F||N_G(F)/F|=|N_G(F)|\\
&\leq |G|.
\end{align*}
Since the Frobenius-Perron dimension of a fusion category is invariant under categorical Morita equivalence \cite[Theorem 2.15]{ENO}, the module category $\M(X,\alpha)$ is pointed if and only if
$$|\Aut_{\vect_{G}^{\omega}}(\M(X,\alpha))|=|G|.$$ The equality is equivalent
to $|\Hom(F,\Tt)|=|F|$, $|N_G(F)|=|G|$ and
$\Aut_G(X)=\Aut_G(X,[\alpha])$. Thus $F$ must be abelian and normal,
and $[\alpha]$  invariant by $\Aut_G(X)$.  The converse also follows from the exact sequence \eqref{sucecion Aut}.
\end{proof}


\begin{thebibliography}{10}

\bibitem{BK}
B.~Bakalov and A.~Kirillov, Jr.
\newblock {\em Lectures on tensor categories and modular functors}, volume~21
  of {\em University Lecture Series}.
\newblock American Mathematical Society, Providence, RI, 2001.

\bibitem{Beggs}
E.~Beggs.
\newblock Making non-trivially associated tensor categories from left coset
  representatives.
\newblock {\em J. Pure Appl. Algebra}, 177(1):5--41, 2003.

\bibitem{Davy2}
A.~Davydov.
\newblock Modular invariants for group-theoretical modular data. {I}.
\newblock {\em J. Algebra}, 323(5):1321--1348, 2010.

\bibitem{Davy}
A.~A. Davydov.
\newblock Galois algebras and monoidal functors between categories of
  representations of finite groups.
\newblock {\em J. Algebra}, 244(1):273--301, 2001.

\bibitem{D}
P.~Deligne.
\newblock Cat\'egories tannakiennes.
\newblock In {\em The {G}rothendieck {F}estschrift, {V}ol.\ {II}}, volume~87 of
  {\em Progr. Math.}, pages 111--195. Birkh\"auser Boston, Boston, MA, 1990.

\bibitem{DPP}
R.~Dijkgraaf, V.~Pasquier, and P.~Roche.
\newblock Quasi {H}opf algebras, group cohomology and orbifold models.
\newblock {\em Nuclear Phys. B Proc. Suppl.}, 18B:60--72 (1991), 1990.
\newblock Recent advances in field theory (Annecy-le-Vieux, 1990).

\bibitem{DGNO}
V.~Drinfeld, S.~Gelaki, D.~Nikshych, and V.~Ostrik.
\newblock On braided fusion categories. {I}.
\newblock {\em Selecta Math. (N.S.)}, 16(1):1--119, 2010.

\bibitem{EG}
P.~Etingof and S.~Gelaki.
\newblock Isocategorical groups.
\newblock {\em Internat. Math. Res. Notices}, (2):59--76, 2001.

\bibitem{EGNO}
P.~Etingof, S.~Gelaki, D.~Nikshych, and V.~Ostrik.
\newblock {\em Tensor categories}, volume 205 of {\em Mathematical Surveys and
  Monographs}.
\newblock American Mathematical Society, Providence, RI, 2015.

\bibitem{ENO}
P.~Etingof, D.~Nikshych, and V.~Ostrik.
\newblock On fusion categories.
\newblock {\em Ann. of Math. (2)}, 162(2):581--642, 2005.

\bibitem{ENO3}
P.~Etingof, D.~Nikshych, and V.~Ostrik.
\newblock Fusion categories and homotopy theory.
\newblock {\em Quantum Topol.}, 1(3):209--273, 2010.
\newblock With an appendix by Ehud Meir.

\bibitem{ENO2}
P.~Etingof, D.~Nikshych, and V.~Ostrik.
\newblock Weakly group-theoretical and solvable fusion categories.
\newblock {\em Adv. Math.}, 226(1):176--205, 2011.

\bibitem{EO}
P.~Etingof and V.~Ostrik.
\newblock Finite tensor categories.
\newblock {\em Mosc. Math. J.}, 4(3):627--654, 782--783, 2004.

\bibitem{Ga1}
C.~Galindo.
\newblock Clifford theory for tensor categories.
\newblock {\em J. Lond. Math. Soc. (2)}, 83(1):57--78, 2011.

\bibitem{Ga2}
C.~Galindo.
\newblock Crossed product tensor categories.
\newblock {\em J. Algebra}, 337:233--252, 2011.

\bibitem{Ga3}
C.~Galindo.
\newblock Clifford theory for graded fusion categories.
\newblock {\em Israel J. Math.}, 192(2):841--867, 2012.

\bibitem{GNN}
S.~Gelaki, D.~Naidu, and D.~Nikshych.
\newblock Centers of graded fusion categories.
\newblock {\em Algebra Number Theory}, 3(8):959--990, 2009.

\bibitem{Justin}
J.~Greenough.
\newblock Monoidal 2-structure of bimodule categories.
\newblock {\em J. Algebra}, 324(8):1818--1859, 2010.

\bibitem{GS}
P.~Grossman and N.~Snyder.
\newblock The {B}rauer-{P}icard group of the {A}saeda-{H}aagerup fusion
  categories.
\newblock {\em Trans. Amer. Math. Soc.}, 368(4):2289--2331, 2016.

\bibitem{IH}
M.~Izumi and H.~Kosaki.
\newblock On a subfactor analogue of the second cohomology.
\newblock {\em Rev. Math. Phys.}, 14(7-8):733--757, 2002.
\newblock Dedicated to Professor Huzihiro Araki on the occasion of his 70th
  birthday.

\bibitem{Tann}
A.~Joyal and R.~Street.
\newblock An introduction to {T}annaka duality and quantum groups.
\newblock In {\em Category theory ({C}omo, 1990)}, volume 1488 of {\em Lecture
  Notes in Math.}, pages 413--492. Springer, Berlin, 1991.

\bibitem{Karpy}
G.~Karpilovsky.
\newblock {\em Clifford theory for group representations}, volume 156 of {\em
  North-Holland Mathematics Studies}.
\newblock North-Holland Publishing Co., Amsterdam, 1989.
\newblock Notas de Matem{\'a}tica [Mathematical Notes], 125.

\bibitem{Lang}
S.~Lang.
\newblock {\em Algebra}, volume 211 of {\em Graduate Texts in Mathematics}.
\newblock Springer-Verlag, New York, third edition, 2002.

\bibitem{MacLane-book}
S.~Mac~Lane.
\newblock {\em Categories for the working mathematician}, volume~5 of {\em
  Graduate Texts in Mathematics}.
\newblock Springer-Verlag, New York, second edition, 1998.

\bibitem{M2}
M.~Mombelli.
\newblock The {B}rauer-{P}icard group of the representation category of finite
  supergroup algebras.
\newblock {\em Rev. Un. Mat. Argentina}, 55(1):83--117, 2014.

\bibitem{Deepak}
D.~Naidu.
\newblock Categorical {M}orita equivalence for group-theoretical categories.
\newblock {\em Comm. Algebra}, 35(11):3544--3565, 2007.

\bibitem{Naidu}
D.~Naidu.
\newblock Categorical {M}orita equivalence for group-theoretical categories.
\newblock {\em Comm. Algebra}, 35(11):3544--3565, 2007.

\bibitem{Naidu-Nik}
D.~Naidu and D.~Nikshych.
\newblock Lagrangian subcategories and braided tensor equivalences of twisted
  quantum doubles of finite groups.
\newblock {\em Comm. Math. Phys.}, 279(3):845--872, 2008.

\bibitem{Nik}
D.~Nikshych.
\newblock Non-group-theoretical semisimple {H}opf algebras from group actions
  on fusion categories.
\newblock {\em Selecta Math. (N.S.)}, 14(1):145--161, 2008.

\bibitem{Nik-surv}
D.~Nikshych.
\newblock Morita equivalence methods in classification of fusion categories.
\newblock In {\em Hopf algebras and tensor categories}, volume 585 of {\em
  Contemp. Math.}, pages 289--325. Amer. Math. Soc., Providence, RI, 2013.

\bibitem{NR}
D.~Nikshych and B.~Riepel.
\newblock Categorical {L}agrangian {G}rassmannians and {B}rauer-{P}icard groups
  of pointed fusion categories.
\newblock {\em J. Algebra}, 411:191--214, 2014.

\bibitem{ostrik}
V.~Ostrik.
\newblock Module categories over the {D}rinfeld double of a finite group.
\newblock {\em Int. Math. Res. Not.}, (27):1507--1520, 2003.

\bibitem{ostrik1}
V.~Ostrik.
\newblock Module categories, weak {H}opf algebras and modular invariants.
\newblock {\em Transform. Groups}, 8(2):177--206, 2003.

\bibitem{Tambara}
D.~Tambara.
\newblock Representations of tensor categories with fusion rules of
  self-duality for abelian groups.
\newblock {\em Israel J. Math.}, 118:29--60, 2000.

\bibitem{Tam-act}
D.~Tambara.
\newblock Invariants and semi-direct products for finite group actions on
  tensor categories.
\newblock {\em J. Math. Soc. Japan}, 53(2):429--456, 2001.

\bibitem{TY}
D.~Tambara and S.~Yamagami.
\newblock Tensor categories with fusion rules of self-duality for finite
  abelian groups.
\newblock {\em J. Algebra}, 209(2):692--707, 1998.

\end{thebibliography}
\end{document}